\documentclass[11pt,a4paper,reqno]{amsart}
\usepackage[english]{babel}
\usepackage[applemac]{inputenc}
\usepackage[T1]{fontenc}
\usepackage{enumitem} 
\usepackage{palatino}
\usepackage{amsmath}
\usepackage{amssymb}
\usepackage{amsthm}
\usepackage{amsfonts}
\usepackage{float}
\usepackage{graphicx}
\usepackage{amsrefs}
\usepackage[colorlinks = true, citecolor = blue]{hyperref}
\title{On the dimension spectrum of infinite subsystems of continued fractions}
\address{University of Connecticut, Department of Mathematics}
\address{University of North Texas, Department of Mathematics}
\author{Vasileios Chousionis}
\author{Dmitriy Leykekhman}
\author{Mariusz Urba\'nski}
\thanks{V.C. is supported by  the Simons Foundation via the project `Analysis and dynamics in Carnot groups', Collaboration grant no.\  521845. D.L.  is supported by the NSF grant no.\ DMS-1522555.   
}
\email{vasileios.chousionis@uconn.edu}
\email{dmitriy.leykekhman@uconn.edu}
\email{urbanski@unt.edu}

\newcommand{\ve}{\varepsilon}

\newcommand{\f}{\phi}
\newcommand{\R}{\mathbb{R}}

\newcommand{\N}{\mathbb{N}}
\newcommand{\Q}{\mathbb{Q}}

\renewcommand{\a}{{\alpha}}
\newcommand{\om}{{\omega}}

\newcommand{\cS}{\mathcal{S}}

\newcommand{\cf}{\mathcal{CF}}

\newcommand{\cH}{\mathcal{H}}

\newcommand{\sg}{\sigma}

\newcommand{\diam}{\operatorname{diam}}

\newcommand{\err}{\operatorname{err}}
\newcommand{\corr}{\operatorname{corr}}

\def\Int{\text{{\rm Int}}}
\def\om{\omega}

\newcommand{\stm}{\setminus}
\newcommand{\ra}{\rightarrow}

\numberwithin{equation}{section}

\newtheorem{thm}{Theorem}[section]
\newtheorem{theorem}{Theorem}[section]

\newtheorem{lm}[thm]{Lemma}
\newtheorem{corollary}[thm]{Corollary}
\theoremstyle{definition}

\newtheorem{propo}[thm]{Proposition}
\theoremstyle{definition}

\newtheorem{defn}[thm]{Definition}
\theoremstyle{definition}

\theoremstyle{definition}
\newtheorem{rem}[thm]{Remark}
\newtheorem{remark}[thm]{Remark}

\begin{document}
\begin{abstract} In this paper we study the dimension spectrum of continued fractions with coefficients restricted to infinite subsets of natural numbers. We prove that if $E$ is any  arithmetic progression, the set of primes, or the set of squares $\{n^2\}_{n \in \N}$, then the continued fractions whose digits lie in $E$ have full dimension spectrum, which we denote by $DS(\cf_E)$. Moreover we prove that if $E$ is an infinite set of consecutive powers then the dimension spectrum $DS(\cf_E)$ always contains a non trivial interval. We also show that there exists some $E \subset \N$ and two non-trivial intervals $I_1, I_2$, such that $DS(\cf_E) \cap I_1=I_1$  and $DS(\cf_E) \cap I_2$ is a Cantor set. On the way we employ the computational approach of Falk and Nussbaum in order to obtain rigorous effective estimates for the Hausdorff dimension of continued fractions  whose entries are restricted to infinite sets.
\end{abstract}

\maketitle 
\section{Introduction}
Given any $E \subset \N$, frequently called an {\em alphabet}, we will denote by $J_E$ the collection of irrationals $x \in (0,1)$ whose continued fraction expansion
$$[e_1(x),e_2(x),\dots]:=\cfrac{1}{e_1(x)+\cfrac{1}{e_2(x)+\cfrac{1}{e_3(x)+\dots}}},$$
has all its coefficients $e_i (x)$ in the set $E$. These naturally defined fractals have been attracting significant attention since the classical paper of Jarnik \cite{jar}, and many authors have investigated their metric and geometric properties \cite{bu1, bu2, cu1, cu2, cu3, FalkRS_NussbaumRD_2016a, hen1, hen2, hentex, jenktexan, JenkinsonO_PollicottM_2001, JenkinsonO_PollicottM_2018, good, MU2, UZ}. We also record that dimension estimates for the sets $J_{E}$ frequently appear in the context of Diophantine approximation; for example they have been employed in the study of the Markoff and Lagrange spectra \cite{bu1, bu2, cufla}, and in a recent important contribution related to Zaremba's conjecture \cite{zar}.

In this paper we will be concerned with the {\em dimension spectrum of continued fractions}. More precisely if $E \subset \N$, the dimension spectrum of continued fractions with coefficients in $E$ is defined as
$$DS(\cf_E)=\{\dim_{\cH}(J_F):F \subset E\},$$
where $\dim_{\cH}(J_F)$ denotes the Hausdorff dimension of the set $J_F$.  Note that when $E$ is finite the dimension spectrum $DS(\cf_E)$ is also finite. However when $E$ is infinite several interesting questions arise concerning the size and structure of $DS(\cf_E)$. Regarding the topological properties of $DS(\cf_E)$, we recently proved in \cite{ChousionisV_LeykekhmanD_UrbanskiM_2018a} that it is always compact and perfect.

If the alphabet $E=\N$ the dimension spectrum $DS(\cf_\N)$ is well understood. Notice, that in this case the limit set $J_\N$ is simply $(0,1) \stm \Q$. In the mid 90s, Hensley \cite{hentex}, and independently Mauldin and the last named author \cite{MU2}, conjectured that the dimension spectrum $DS(\cf_{\N})$ is full, that is $$DS(\cf_{\N})=[0, \dim_{\cH}(J_{\N})]=[0,1].$$
This conjecture is frequently referred to as the {\em Texan conjecture}, since Hensley and the authors of \cite{MU2} were affiliated with Texan institutions at that time. In 2006, Kesseb\"ohmer and Zhu \cite{KessebohmerM_ZhuS_2006} gave a positive answer to the Texan conjecture, and developed several influential tools for the study of the dimension spectrum in a more general context. See also \cite{MU1, jenktexan} for some earlier partial results.

Beside the case when $E=\N$, the dimension spectrum of continued fractions with digits restricted to infinite subsets $E \subset \N$ has not been investigated. In this paper, for the first time, we study the dimension spectrum $DS(\cf_E)$ for various sets $E \subset \N$. We will prove that even when $E$ is arbitrarily sparse the dimension spectrum $DS(\cf_{E})$ can be full. Our first main result reads as follows.
\begin{thm}
\label{progrespc}
Let $E \subset \N$ be any infinite arithmetic progression. Then
$$DS(\cf_{E})=[0, \dim_{\cH}(J_E)].$$
\end{thm}
Moreover we will show that the continued fractions whose partial quotients are prime numbers have full dimension spectrum. 
\begin{thm}
\label{primes}Let $E$ be the set of primes. Then
$$DS(\cf_{E})=[0, \dim_{\cH}(J_E)].$$
\end{thm}
And the same holds true for continued fractions whose digits are squares.
\begin{thm}
\label{squares}Let $E$ be  the set of squares $\{n^2\}_{n \in \N}$. Then
$$DS(\cf_{E})=[0, \dim_{\cH}(J_E)].$$
\end{thm}
One can then naturally ask how the dimension spectrum is affected when the size of the gaps in $E$ increases exponentially. Our next theorem asserts that when the alphabet consists of consecutive powers, i.e.  when $E=\{\lambda^n\}_{n \in \N}$ for some $\lambda \in \N, \lambda \geq 2$, its dimension spectrum always contains a non-trivial interval.
\begin{thm}
\label{powerintro} Let $\lambda \in \N, \lambda \geq 2,$ and let $E_{\lambda}=\{\lambda^n\}_{n \in \N}.$ Then there exists some $s(\lambda)>0$ such that 
$$[0, \max\{s(\lambda), \dim_{\cH}(J_{E_{\lambda}})\}] \subset DS(\cf_{E_{\lambda}}).$$
\end{thm}
We currently do not know if $E_{\lambda}$ defines continued fractions with full spectrum. We are inclined to believe that this is not the case, because as our next theorem shows, there exist continued fractions  with alphabets consisting of scaled powers whose dimension spectrum is not full. Interestingly enough, their spectrum also contains a nowhere dense part.
\begin{thm}
\label{50100intro} 
Let $E=\{2\cdot 100^{n-1}\}_{n \in \N}$. Then there exist $ 0<s_1<s_2<\dim_{\cH}(J_{E}),$ such that
\begin{enumerate}[label=(\roman*)]
\item \label{inter} $[0,s_1] \subset DS(\cf_{E})$,
\item \label{nowden} $[s_2, \dim_{\cH}(J_{E})] \cap DS(\cf_{E})$ is nowhere dense.
\end{enumerate}
\end{thm}
Since  $DS(\cf_E)$ is always compact and perfect \cite{ChousionisV_LeykekhmanD_UrbanskiM_2018a}, we deduce that if $E$ is as in the previous theorem, then there exists an interval $I$ such $DS(\cf_E) \cap I$ is a Cantor set. Moreover we can prove that if $E$ is a lacunary sequence then there exists an interval $I$ such that $DS(\cf_E) \cap I$ is a Cantor set, see Theorem \ref{lacspec}. It becomes evident that if the alphabet $E$ has gaps whose size increases exponentially (or faster), then the dimension spectrum has a very intriguing structure. Our methods do not seem to be sufficient for fully describing the dimension spectrum of continued fractions with such alphabets and we consider that new ideas will be needed for such task.

Our proofs employ the machinery which we developed recently in \cite{ChousionisV_LeykekhmanD_UrbanskiM_2018a} in order to show that complex continued fractions have full spectrum, as well as some key ideas of Kesseb\"ohmer and Zhu from \cite{KessebohmerM_ZhuS_2006}. All the techniques used, depend on a well known remarkable feature of the sets $J_E$; they can be realized as attractors of iterated function systems consisting of conformal maps; see Section \ref{sec:prelim} for more details.  This allows us to take advantage of the technology of {\em thermodynamic formalism} developed in \cite{MU1, MauldinRD_UrbanskiM_2003a}. More precisely,  topological pressure and Perron-Frobenius operators are used extensively in our approach.   

In addition, we depend crucially on Hausdorff dimension estimates for the sets $J_E$. 
We stress that the topic of estimating the Hausdorff dimension of the sets $J_E$ has attracted significant attention, see e.g.  \cite{jar, bu1, bu2, good, FalkRS_NussbaumRD_2016a, hen1, JenkinsonO_PollicottM_2018, good, HeinemannSM_UrbanskiM_2002a}. In particular when the alphabet $E$ is finite and rather small, very sharp estimates are known.  For example, quite recently Jenkinson and Pollicott  \cite{JenkinsonO_PollicottM_2018} were able to approximate the Hausdorff dimension of $J_{\{1,2\}}$ with an accuracy of over $100$ decimal places.

 In the present paper we adopt the computational approach of Falk and Nussbaum  \cite{FalkRS_NussbaumRD_2016a, FalkRS_NussbaumRD_2016b} and we obtain rigorous Hausdorff dimension estimates for continued fractions with infinite alphabets. We consider that these estimates  are interesting on their own right, and as mentioned earlier they also play a crucial for the study of the dimension spectrum. To the best of our knowledge, the only previous attempt of estimating the Hausdorff dimension of continued fractions with infinite alphabet was in \cite{HeinemannSM_UrbanskiM_2002a}, were the case of $J_{2 \N}$ was considered. We estimate the Hausdorff dimension of the sets  $J_E$ for various infinite alphabets $E$, such as the odds, the evens (improving the estimate from \cite{HeinemannSM_UrbanskiM_2002a}), the primes, the squares, the powers of $2$, etc. We gather some of these estimates in Table \ref{tableintro}. These estimates can definitely be improved with heavier computations. 
\begin{table}[ht]
\label{tableintro}
\caption{Hausdorff dimension estimates for various sets $J_E$}
\centering
\vspace{.1in}
\begin{tabular}{|l|c|c|c|c|}
\hline
Subsystem           & Hausdorff dim. interval & size of the intervals  \\
\hline
$E_{odd}=\{1,3,5,\dots\}$ & $[0.821160, 0.821177]$  & $1.7e-05$ \\
\hline
$E_{even}=\{2,4,6,\dots\}$ & $[0.71936, 0.71950]$ & $1.4e-04$ \\
\hline
$E_{1mod3}=\{1,4,7,\dots\}$ & $[0.743520, 0.743586]$ & $6.6e-05$\\
\hline
$E_{2mod3}=\{2,5,8,\dots\}$ & $[0.66490, 0.66546]$ & $5.6e-04$ \\
\hline
$E_{0mod3}=\{3,6,9,\dots\}$ & $[0.63956, 0.64073]$ & $1.2e-03$\\
\hline
$E_{prime}=\{2,3,5,\dots\}$ & $[0.67507, 0.67519]$ & $1.2e-04 $ \\
\hline
$E_{square}=\{1,2^2,3^2,\dots\}$ & $[0.59825575, 0.59825579]$ & $4.0e-08 $ \\
\hline
$E_{power2}=\{2,2^2,2^3,\dots\}$ & $[0.4720715327, 0.4720715331]$ & $4.0e-10$ \\
\hline
$E_{power3}=\{3,3^2,3^3,\dots\}$ & $[0.3105296859, 0.3105296860]$ & $ 1.0e-10$ \\
\hline
$E_{lac}=\{2,2^{2^2},2^{3^2},\dots\}$ & $[0.2362689121, 0.2362689123]$ & $ 2.0e-10$  \\
\hline
\end{tabular}
\label{table: Dim estimates}
\end{table} 

The paper is organized as follows. In Section \ref{sec:prelim} we introduce all the relevant concepts related to conformal iterated function and their thermodynamic formalism. In Section \ref{sec:dim} we establish rigorous Hausdorff dimension estimates for continued fractions whose entries are restricted to infinite sets using the approach of Falk and Nussbaum; and we also include a careful description of the Falk-Nussbaum method for convenience of the reader. Finally in Section \ref{sec:spec} we prove Theorems \ref{progrespc}, \ref{primes}, \ref{squares}, \ref{powerintro}, and \ref{50100intro}.

\textbf{Acknowledgements.} We would like to thank De-Jun Feng for some useful discussions during the preparation of this paper.

\section{Preliminary results}\label{sec:prelim}

Let $(X,d)$ be a compact 
metric space and let $E$ be a countable set with at least two elements. Let 
$\cS=\{\f_e:X\to X:e\in  
E\}$ be a  
collection of uniformly contracting injective maps. Any such 
collection $\cS$  is called an {\em iterated function system} (IFS). We denote by $E^*=\cup_{n\ge
1}E^n$ the set of finite words from $E$. For any word $\om\in E^n$, $n\ge 1$, we set
$$
\f_\om=\f_{\om_1}\circ\f_{\om_2}\circ\dots\circ \f_{\om_n}.
$$
If $\om\in E^*\cup E^\N$ and $n\in \N$ does not exceed the length of $\om$, 
we denote by $\om|_n$ the word $\om_1\om_2\dots\om_n$. Since $\cS$ consists of uniformly contracting maps, there exists some $s \in (0,1)$ such that
$$d(\f_e(x), \f_e(y)) \leq s d(x,y)$$
for every $e \in E$ and every pair $x, y \in X$. Observe that if
$\om\in E^\N$ then
$$
\diam(\f_{\om|_n}(X))\le s^n\diam(X). 
$$Hence  $\{\f_{\om|_n}(X)\}_{n \in \N}$ is a decreasing sequence of compact sets whose diameters converge to zero. Therefore 
\begin{equation}
\label{natproj}
\pi(\om)=\bigcap_{n=1}^\infty\f_{\om|_n}(X)
\end{equation}
is a singleton and \eqref{natproj} defines a coding map $\pi:E^\N\to X$. The set 
$$J_{\cS}=\pi(E^\N)=\bigcup_{\om\in E^\N}\bigcap_{n=1}^\infty\f_{\om|n}(X),$$ is called 
the {\em limit set} (or {\em attractor}) associated to the system $\cS$. 

Let
$
\sg: E^\mathbb{N} \ra E^\mathbb{N}
$
be the  shift map,  which  is given by the
formula
$$
\sg\left( (\om_n)^\infty_{n=1}  \right) =  \left( (\om_{n+1})^\infty_{n=1}  \right).
$$
Note that the shift map simply discards the first coordinate. We record  that for any $e \in E$ and any $\om \in E^\N$,
$\f_e(\pi(\om))=\pi(e\om)$ and
$\pi(\om)=\f_{\om_1}(\pi(\sg(\om)))$. Hence the limit set $J_{\cS}$ is {\em invariant} for $\cS$, that is
$$
J_{\cS}=\bigcup_{i\in I}\f_i(J_\cS). 
$$
Notice also that if $E$ is finite, then $J$ is compact. 

\begin{defn}
\label{cifs}
An iterated function system $\cS=\{\f_e:X\to X:e\in  
E\}$ is called
conformal (CIFS) if $X$ is a compact connected subset of the Euclidean space 
${\Bbb R}^d$ and the following conditions are satisfied.
\begin{enumerate}[label=(\roman*)]
\item \label{cfi} $X=\overline{\Int (X)}$
\item ({\em Open set condition} or OSC) For all $a,b \in E, a \neq b,$
$$\f_a (\Int(X)) \cap \f_b(\Int(X))=\emptyset.$$
\item There exists an open connected set $X\subset V\subset {\Bbb R}^d$ such that
all maps $\f_e$, $e\in E$, extend to $C^{1+\ve}$ diffeomorphisms on $V$ and are conformal on $V$.  
\item \label{cfbdp} ({\em Bounded Distortion Property} or BDP) There exists $K\ge 1$ 
such that $$|\f_\om'(y)|\le K|\f_\om'(x)|$$ for every $\om\in E^*$ and every 
pair of points $x,y\in X$, where $|\f_\om'(x)|$ denotes the norm of the derivative. 
\end{enumerate}
\end{defn}
We will denote the {\em best distortion constant} of $\cS$ by
$$K_{\cS}=\sup_{\om \in E^\ast} \left\{\max_{x,y \in X} \frac{|\f_\om'(y)|}{|\f_\om'(x)|}.\right\}$$
For $t\ge 0$ and $n \in \N$ we define
\begin{equation*}
Z_{n}(E,t):=Z_n(t) = \sum_{\om\in E^n} \|\phi'_\om\|^t_\infty,
\end{equation*}
where for any $\om \in E^*$ we use the notation
$\| \f'_\om\|_\infty := \max_{x \in X} |\f'_\om (x)|.$
It follows that
$
Z_{m+n}(t)\le Z_m(t)Z_n(t),
$
and consequently, the
sequence $(\ln Z_n(t))_{n=1}^\infty$ is subadditive. Thus, 
$$
\lim_{n \to  \infty}  \frac{ \ln Z_n(t)}{n}
=\inf_{n \in \N} \left(\frac{\ln Z_n(t)}{n}\right)<\infty.$$ The value of
the limit is denoted by $P(t)$ or, if we want to be more precise, by
$P_E(t)$ or $P_\mathcal{S}(t)$. It is called the {\em topological pressure} of the system $\mathcal{S}$ evaluated at the parameter $t$. The topological pressure provides an indispensable tool in the dimension study of the limit set of $\cS$. Before making the previous statement precise we summarize a few basic properties of the pressure function, whose proofs can be found in \cite{MU1}.
\begin{propo}\label{p2j85}
If $\cS$ is a conformal IFS then the
following conclusions hold.
\begin{enumerate}[label=(\roman*)]
\item \label{z1pfin}  $\{t\ge 0:Z_1(t)< +\infty\}=\{t\ge 0:P(t)< +\infty\}$.
\item \label{presscontconv} The topological pressure $P$ is
strictly  decreasing  on $[0,+\infty)$ with $P(t) \to -\infty$ as
$t\to+\infty$. Moreover, the function
$P$ is convex and  continuous on  $\overline{\{t\ge 0:Z_1(t)< +\infty\}}$.
\item \label{p0infty} $P(0)=+\infty$  if and only if $E$  is infinite.
\end{enumerate}
\end{propo}
It becomes evident that the pressure function has two critical parameters:
$$\theta_\cS:=\inf \{t\ge 0:P(t)< +\infty\}, \mbox{ and }
h_{\mathcal{S}} := \inf\{t\geq 0:  P(t)\leq 0\}.
$$
In particular $h_\cS$ is known as  {\it Bowen's parameter} and is related to the Hausdorff dimension of the limit set, as the following theorem asserts. We record that if $\cS=\{\f_e\}_{e \in E}$ is a conformal IFS and $F \subset E$ we denote the limit set of $\cS_{F}:=\{\f_e\}_{e \in F}$ by $J_F$.
\begin{thm}[\cite{MU1} ]\label{t1j97}
If $\cS=\{\f_e\}_{e \in E}$ is a conformal iterated function system, then
$$
h_\mathcal{S}
= \dim_{\mathcal{H}}(J_\mathcal{S})
= \sup \{\dim_\cH(J_F):  \, F \subset E \, \mbox{finite} \, \}.
$$
\end{thm}
When the conformal IFS is {\em strongly regular}, i.e. when there exists some $t  \geq 0$ such that $P(t) \in (0, +\infty)$, it is possible to obtain asymptotic information for the Hausdorff dimension of subsystems. The following theorem was originally obtained in \cite{HeinemannSM_UrbanskiM_2002a} and very recently it was generalized via a simplified proof in \cite{ChousionisV_LeykekhmanD_UrbanskiM_2018a}. Before stating it, we record that if $\cS=\{\f_e\}_{e \in E}$ is a conformal IFS the {\em characteristic Lyapunov exponent} of $\cS$ is
\begin{equation}
\label{lyap}
\chi(\cS)=-\int_{E^{\N}}\ln\|\phi'_{\om_1}(\pi(\sg(\om))\| \,d \tilde{\mu}_h(\om),
\end{equation}
where $\tilde{\mu}_h$ is the unique shift-invariant ergodic measure on $E^\N$, globally equivalent to the conformal measure of $E^\N$. We will only use the fact that $\tilde{\mu}_h$ is a probability measure, and for this reason we will not define it explicitly. For more information, we refer the reader to \cite{CTU, MauldinRD_UrbanskiM_2003a}. If $F \subset E$, we denote by $h_{F}$ the Bowen's parameter of the subsystem $\cS_{F}=\{\f_e\}_{e \in F}$. Note that by Theorem \ref{t1j97} it holds that $h_{F}=\dim_{\cH} (J_{F})$.
\begin{thm}[\cite{HeinemannSM_UrbanskiM_2002a, ChousionisV_LeykekhmanD_UrbanskiM_2018a}]
\label{hein}
Let $\cS=\{\f_e\}_{e \in E}$ be a  strongly regular conformal IFS. If $F\subset E$ is a finite set such that $h_{F} \geq \theta_{\cS}$, then 
$$
\dim_{\cH} (J_{\cS})-\dim_{\cH} (J_{F}) \leq  \frac{K^{h_F}} {\chi_{\cS}}\,\sum_{E \stm F} \, \|\f'_e\|^{h_F}_{\infty}.
$$
\end{thm}

From now on we are going to focus our attention on  systems generated by the conformal maps $\f_e(x)=\frac{1}{e+x}$ for $e \in \N$. In particular we are going to consider the systems
\begin{equation}
\label{cfedef}
\cf_E:= \left\{\f_{e}:[0,\min{E}^{-1}] \ra [0,\min{E}^{-1}]: \f_e(x)=\frac{1}{e+x} \mbox{ for }e \in E \right\},
\end{equation}
for $E \subset \N$. Indeed it is easy to check that if $e \in E$ then $\f_e ([0,\min{E}^{-1}] ) \subset [0,\min{E}^{-1}]$
and $\cf_E$ satisfies the properties \ref{cfi}-\ref{cfbdp} of Definition \ref{cifs}, see e.g. \cite{MU2}. Formally if $1 \notin E$, $\{\f_e \}_{e \in E}$ is not a conformal IFS because $\f'_1(0)=-1$. Nevertheless this is not a real problem since the system of second level maps $\{ \f_e \circ \f_j:(e,j) \in E \times E\}$ has the same limit set as $\cf_{E}$ and it is uniformly contractive. We are going to denote the limit set of $\cf_E$ by $J_{E}$. Notice that the limit set $J_E$ is the set of all irrational numbers in $(0,1)$ whose continued fractions expansion only contains digits from $E$.

The following proposition  provides an upper bound for the best distortion constant for the systems $\cf_E$ when $E \subset \N \stm \{1\}$. When $1 \in E$ it is easy to see that $K_{\cf_{E}}=4$, see e.g. \cite{MU2}.
\begin{propo}\label{prop: distortion}
\label{distco} Let $E \subset \{k, k+1, k+2, \dots\}$ for some $k \in \N, k \geq 2$. Then
$$K_{E}:=K_{\cf_{E}} \leq \exp \left( \frac{2}{k^2-1}\right).$$
\end{propo}
\begin{proof}
Let $I_k=[0,1/k]$. Recalling \eqref{cfedef}, $\f_{e} (I_k) \subset I_k$ for all $e \in E$. Notice that $\|\f'_e\|_\infty \leq k^{-2}$ for all $e \in E$, hence 
$\|\f'_{\om}\|_\infty \leq k^{-2 |\om|}$
for all $\om \in E^{\ast}$. Therefore by the mean value theorem,
\begin{equation}
\label{diamk}
\diam(\f_{\om}(I_k)) \leq k^{-2|\om|-1}.
\end{equation}
If $e \in E$ and $x, y \in I_k$,
\begin{equation}
\label{logdif1}
\begin{split}
\big|\ln |\f'_e(y)|-\ln |\f'_e(x)|\big|&=2|\ln(e+y)-\ln(e+x)|  \\
& \leq \frac{2}{\min \{e+y,e+x\}}|y-x| \leq \frac{2}{k} |y-x|.
\end{split}
\end{equation}
Let $q \in \N$ and $\om \in E^{q}$. For any $z \in I_k$ let
$$z_j(\om)=\begin{cases}\f_{\om_{j+1}} \circ \dots \circ \f_{\om_{q}}(z), \quad \mbox{ if }j=1, \dots, q-1 \\
z \quad \mbox{ if }j=q.
\end{cases}$$
 Then for all $x,y \in I_k$,
 \begin{equation*}
\label{logdif2}
\begin{split}
\big|\ln |\f'_\om(y)|-\ln |\f'_\om(x)|\big|&\leq \sum_{j=1}^q \left| \ln|\f'_{\om_j}(x_j(\om))|-\ln|\f'_{\om_j}(y_j(\om))| \right| \\
&\overset{\eqref{logdif1}}{\leq} \frac{2}{k} \sum_{j=1}^q |x_j(\om)-y_j(\om)|\overset{\eqref{diamk}}{\leq} \frac{2}{k} \sum_{j=1}^q k^{-2(q-j)-1} \\
& \leq \frac{2}{k^2} \sum_{i=0}^\infty k^{-2i}=\frac{2}{k^2-1}.
\end{split}
\end{equation*}
Hence for all $x,y \in I_k$ and every $\om \in E^{\ast}$ we get that
$$\frac{|\f'_{\om}(x)|}{|\f'_{\om}(y)|} \leq \exp \left(\frac{2}{k^2-1}\right).$$
The proof is complete.
\end{proof}

We will now provide lower bounds for the Lyapunov exponent of the systems $\cf_{E}$ for $E \subset \N$.
\begin{propo}\label{prop: Lyapunov}
\label{lyapbound} Let $E \subset \{k, k+1, k+2, \dots\}$ for some $k \in \N$. Then
$$\chi_E:=\chi(\cf_{E}) \geq 2 \ln \left( \frac{k+\sqrt{k^2+4}}{2} \right).$$
\end{propo}
\begin{proof} For any $\om \in E^{n}$ we define 
\begin{equation}
\label{recur}q_{-1}(\om)=0,\, q_0(\om)=1 \mbox{ and }q_i(\om)=\om_i q_{i-1}(\om)+q_{i-2}(\om) \mbox{ for }i=1, \dots, n.
\end{equation}
It is not difficult to see, see e.g. \cite[Lemma 4.1]{KessebohmerM_ZhuS_2006}, that for all $\om \in \N^{n}, n \in \N,$ and $x \in [0,1]$
\begin{equation}
\label{recurder}
\f'_{\om}(x)=\frac{(-1)^n}{q_n(\om)^2\left( 1+x \frac{q_{n-1}(\om)}{q_n(\om)}\right)^2}.
\end{equation}
Therefore  for all $\om \in \N^{n}$,
$$\|\f'_{\om}\|_{\infty}=|\f'(0)|=q_n(\om)^{-2}.$$
Now consider the sequence $q_n^\ast=q_n(k^n)$ where $k^n:=k\dots k \in E^n$ .  It is easy to see using \eqref{recur} and strong induction that for all $n\in \N$ and every $\om \in E^n$, 
$q_n(\om) \geq q_n^{\ast}.$ Hence for all $\om \in E^n$,
\begin{equation}
\label{derboundrec}
\|\f'_{\om}\|_{\infty} \leq {{q_n^\ast}}^{-2}.
\end{equation}

Let $\zeta: E^\N \ra \R$ defined by 
$$\zeta(\om)=\ln\|\phi'_{\om_1}(\pi(\sg(\om))\|,$$
and note that if $\om \in E^\N$, then
\begin{equation*}\begin{split}
\sum_{j=0}^{n-1}  \zeta \circ \sg^j(\om)&=\sum_{j=0}^{n-1} \log \| \f'_{\om_{j+1}}(\pi(\sg^{j+1}(\om)))\|\\
&=\log\left(\| \f'_{\om_{1}}(\pi(\sg^{1}(\om)))\|\,\| \f'_{\om_{2}}(\pi(\sg^{2}(\om)))\|\cdots \|\f'_{\om_{n}}(\pi(\sg^{n}(\om)))\| \right).
\end{split}\end{equation*}
By the Leibniz rule and the fact that  $\pi(\sg^k(\om))= \f_{\om_{k+1}}\circ \dots \circ \f_{\om_n}(\pi(\sg^n(\om)))$ for $1 \leq k \leq n$ we deduce that
\begin{equation}
\label{logleib}
\sum_{j=0}^{n-1} \zeta \circ \sg^j(\om)=\log \| \f'_{\om|_n}(\pi(\sg^n(\om)))\|.
\end{equation}
Recalling \eqref{lyap} and using the fact that $\tilde{\mu}_h$ is a shift invariant measure we deduce that
$$\chi_{E}=-\frac{1}{n} \int_{E^\N} \log \|\f'_{\om|_n}(\pi(\sg^n(\om)))\| d \tilde{\mu}_h.$$
Therefore, \eqref{derboundrec} implies that for all $n \in \N$,
\begin{equation}
\label{lyap1}
\chi_E \geq \frac{2}{n} \ln q_n^{\ast}.
\end{equation}

According to \eqref{recur}, the sequence $(q_n^\ast)_{n \in \N}$ is given by the following recursive formula,
\begin{equation}
\label{recur1}q^{\ast}_{-1}=0, \,q^{\ast}_0(\om)=1 \mbox{ and }q^{\ast}_i(\om)=k q^{\ast}_{i-1}(\om)+q^{\ast}_{i-2}(\om) \mbox{ for }i \in \N.
\end{equation}
This is a linear recurrence relation with constant coefficients, hence 
$$q^{\ast}_n=c_1 r_1^n+c_2r_2^n$$
where $c_1,c_2 \in \R \stm \{0\}$ and
$$r_1=\frac{k+ \sqrt{k^2+4}}{2}, \quad r_2=\frac{k- \sqrt{k^2+4}}{2}.$$
We record that $r_1,r_2$ are solutions to the characteristic equation $r^2-kr-1=0$ of the recursive relation. The coefficients $c_1,c_2$ can be computed explicitly by the initial conditions, nevertheless we will only need the fact that $c_1>0$. This follows because $q^{\ast}_n >0$ for all $n \in \N$ and $r_2<0$. Note also that $r_1 > |r_2|$. Therefore for $n$ large enough
$$\left|c_2 \left( \frac{r_2}{r_1}\right)^n \right|< \frac{c_1}{2}.$$
Since $$q^{\ast}_n=r_1^n \left(c_1+c_2 \left( \frac{r_2}{r_1}\right)^n\right),$$
we deduce that for $n$ large enough,
$$\frac{c_1}{2}r_1^n \leq q^{\ast}_n \leq \frac{3c_1}{2}r_1^n.$$
Using \eqref{lyap1} and letting $n \ra +\infty$ we deduce that
$$\chi_E \geq 2 \ln r_1=2 \ln \left( \frac{k+\sqrt{k^2+4}}{2} \right).$$
The proof is complete.
\end{proof}
\begin{propo}
\label{bijpropo} Let $E,F \subset \N$ and suppose that there exists an increasing bijection $s: E \ra F$. Then
$$\dim_{\cH}(J_F) \leq \dim_{\cH}(J_E).$$
\end{propo}
\begin{proof} Recalling \eqref{recurder} from the proof of Proposition \ref{lyapbound} we have that 
$$\|\f'_{\om}\|_\infty=q_{|\om|}^{-2}(\om)$$
for all $\om \in \N^\ast$, where the sequence $q_n(\om)$ was defined in \eqref{recur}. Now if $\om \in E^n$ we define a mapping $S_n:E^n \ra F^n$ by
$$S_n(\tau)=(s(\om_k))_{k=1}^n, \quad \om \in E^n.$$
Note that $S_n$ is a bijection and $\om_k \leq (S_n(\om))_k$ for all $k=1,\dots,n$. It then follows easily by induction (using the recursive properties of $q_n$) that 
$$q_n(\om) \leq q_n(S_n(\om)), \quad \mbox{ for all }\om \in E^n.$$
Therefore for all $t \geq 0$, we have
\begin{equation*}
\begin{split}
Z_n(F,t)&=\sum_{\om \in E^n} \|\f'_{S_n(\om)}\|^{t}=\sum_{\om \in E^n} q^{-2t}_{n}(S_n(\om)) \leq \sum_{\om \in E^N} q^{-2t}_{n}(\om)\\
&=\sum_{\om \in E^N} \|\f'_{\om}\|^t_\infty=Z_n(E,t).
\end{split}
\end{equation*}
Hence $P_{F}(t) \leq P_E(t)$ and Theorem \ref{t1j97} implies that
$$\dim_{\cH}(J_F)=\inf \{t \geq 0: P_F(t) \leq 0\} \leq \inf \{t \geq 0: P_{E}(t)\leq 0\}=\dim_{\cH}(J_E).$$
The proof is complete.
\end{proof}

\section{Hausdorff dimension estimates for continued fractions with restricted entries}\label{sec:dim}

In this section we will establish rigorous Hausdorff dimension estimates for continued fractions whose entries are restricted to infinite sets. Although currently there are plenty of good and quite fast algorithms to calculate the Hausdorff dimension of  limit sets of conformal IFSs with high accuracy, for example \cite{McMullenC_1998, JenkinsonO_PollicottM_2001, JenkinsonO_PollicottM_2018, FalkRS_NussbaumRD_2016a}, to the best of our knowledge only the approach of Falk-Nussbaum in  \cite{FalkRS_NussbaumRD_2016a, FalkRS_NussbaumRD_2016b}
gives rigorous upper and lower bounds with good accuracy for  IFSs with arbitrary large alphabet. In this paper we adopt their computational approach  supplemented by the  interval arithmetic software \texttt{IntLab} for verification of the float point computations \cite{RumpS_2010}. Since the computational method of Falk and Nussbaum plays an important role in our numerical experiments, we will now describe their approach in more detail. 

\subsection{The Method of Falk and Nussbaum} 

The method developed in \cite{FalkRS_NussbaumRD_2016a, FalkRS_NussbaumRD_2016b} is quite general, but for simplicity, and since in the present paper we are only concerned with continued fractions, we are going to explain how it applies to the systems $\cf_E$, defined in \eqref{cfedef}.
A key ingredient in \cite{FalkRS_NussbaumRD_2016a, FalkRS_NussbaumRD_2016b} is the finite dimensional approximation of a parametric family of {\em transfer operators}. In the case of the continued fractions IFS $\cf_E$, where $E=\{e_i\}_{i=1}^P$ is a finite set of positive integers, the transfer operators  $L_t:C^k([0,1])\to C^k([0,1]), t >0, k \in \N$,  take the form
\begin{equation}\label{eq: Perron-Frobenius}
 (L_tf)(x)=\sum_{i=1}^{P}|{\phi'}_{e_i}(x)|^{t}f({\phi}_{e_i}(x))=\sum_{i=1}^P{|x+{e_i}|^{-2t}}f\left(\frac{1}{x+{e_i}}\right).
\end{equation}
We start by restating \cite[Theorem 3.1]{FalkRS_NussbaumRD_2016b} in the case of $\cf_E$. 
\begin{theorem}\label{thm: Falk_Nuss 1} 
Let $E=\{e_i\}_{i=1}^P$ be a finite set of positive integers and let $\cf_E$ be the corresponding continued fractions IFS as in \eqref{cfedef}. For all $t >0$, $L_t:C^k([0,1])\to C^k([0,1])$ has a unique strictly positive eigenvector $\rho_t$ with $L_t \rho_t=\lambda_t \rho_t$,  where $\lambda_t>0$ and $\lambda_t=r(L_t)$ is the  spectral  radius  of $L_t$. Furthermore,  the  map $t\to \lambda_t$ is  strictly  decreasing  and  continuous,  and  for  all $1\leq p \leq k$,
\begin{equation}
\label{thm31der}
(-1)^p \rho^{(p)}_t(x)>0\quad \text{for all} \quad x\in[0,1]
\end{equation}
and
\begin{equation}
\label{thm31normder}
|\rho^{(p)}_t(x)|\le (2 t)(2t+ 1)\cdots(2t+p-1)(k^{-p})\rho_t(x),
\end{equation}
where $k= \min_j{e_j}$.  Finally, the Hausdorff dimension of the limit set $J_E$ is the unique parameter $h$ such that $\lambda_h=1$.
 \end{theorem}
 
\begin{rem}
Due to its crucial role in the numerical approach of Falk and Nussbaum, we will now give a brief history of Theorem \ref{thm: Falk_Nuss 1}. Although a proof can be found in \cite{FalkRS_NussbaumRD_2016b}, we would like to remark that several parts of this theorem have been established in stronger forms in the past using different methods. More precisely:
\begin{enumerate}
\item The proofs  of \cite[Theorem~6.1.2 and Corollary~6.1.4]{MauldinRD_UrbanskiM_2003a} (carried on therein for $t$ being such that $\lambda_t=1$) yield the functions $\rho_t$ to be real-analytic and not merely $C^{k}$. In fact these proofs give that real analyticity of $\rho_t$ holds for: 
\begin{enumerate}
\item all conformal IFSs, with any countable alphabet $F$, and not merely those generated by a continued fraction algorithm with a finite alphabet, and
\item all $t\in \{s>0:P_F(s)<+\infty\}$ (which is equal to $(0,+\infty)$ if $E$ is finite),
\end{enumerate}

\,\item The fact that the map $t\to \lambda_t$ is strictly decreasing is known in the general framework of thermodynamic formalism since 1978 when it was proved by Ruelle in \cite{ruelle}. For the realm of conformal IFSs with any countable alphabet $E$ (and not merely for IFSs generated by a continued fraction algorithm with a finite alphabet) see  \cite[Theorem~2.6.13]{MauldinRD_UrbanskiM_2003a} and the references therein.

\,\item It follows by  \cite[Theorem~{2.6.12}]{MauldinRD_UrbanskiM_2003a} that the map $t\mapsto \lambda_t$ is real-analytic and not merely continuous. In fact, this is proven in \cite{MauldinRD_UrbanskiM_2003a} for all conformal IFSs, with any countable alphabet $F$. 

\,\item The fact that the Hausdorff dimension of the limit set $J_E$ is the unique parameter $h$ such that $\lambda_h=1$ is known for all conformal IFSs, with any finite alphabet $F$, since 1988 when it was proved by Bedford in \cite{bedford}.  
 The full version of this formula, holding for all conformal IFSs with any countable alphabet $F$, was proved in \cite[Theorem~3.15]{MU1}, see also  \cite[Theorem~4.2.13]{MauldinRD_UrbanskiM_2003a}.

\,\item The facts that $\frac{\partial}{\partial x}\rho_t(x)\le 0$ and $\frac{\partial^2}{\partial x^2}\rho_t(x)\ge 0$ for all $x\in[0,1]$ follow immediately from the proof of \cite[Theorem~2.4.3]{MauldinRD_UrbanskiM_2003a} and the fact that all the functions $[0,1]\ni x\mapsto |\phi_\om'(x)|$, $\om\in E^*$, generated by the system $\cf_E$, are decreasing and convex. 

\,\item To the best of our knowledge,  \eqref{thm31normder} and its proof first appears in the work of Falk and Nussbaum 
\cite{FalkRS_NussbaumRD_2016b}.
\end{enumerate} 
\end{rem}

Note that applying \eqref{thm31der} for $p=1$ and $p=2$  we get that $\rho_t(x)$ is decreasing and convex on $[0,1]$.  As an immediate consequence it also holds that
\begin{equation}\label{eq: property of vs}
\rho_t(x_2)\le \rho_t(x_1)e^{\frac{2t}{k}|x_2-x_1|},\quad \mbox{ for all } x_1,x_2\in [0,1].
\end{equation}

\begin{rem}
In fact, it already follows from the proofs of  \cite[Theorem~6.1.2 and Corollary~6.1.4]{MauldinRD_UrbanskiM_2003a} (carried on therein for $t$ being such that $\lambda_t=1$) and Proposition \ref{prop: distortion} that 
$$
\rho_t(x_2)\le \rho_t(x_1)\exp\left(\frac{2t}{k^2-1}|x_2-x_1|\right).
$$
Moreover  a similar estimate holds for all conformal IFSs (in fact GDMSs) with a countable alphabet and not merely for those generated by a continued fraction algorithm with a finite alphabet.
\end{rem}
Since $P$ can be arbitrary large, for computational reasons we first truncate the operator $L_t$.  We select $M< P$  and define
\begin{equation}\label{eq: Perron-Frobenius}
 (L_{M,t} f)(x)=\sum_{i\le M}{|x+{e_i}|^{-2t}}f\left(\frac{1}{x+{e_i}}\right).
\end{equation}
Since the eigenvector $\rho_t$ is positive, we have
$$
(L_{M,t}\rho_t)(x)\le (L_t\rho_t)(x), \quad \mbox{for all } x\in[0,1].
$$
To obtain an upper bound we notice that  \eqref{eq: property of vs} implies that for any $y\in[0,1]$,
$$
\rho_t(y)\le \rho_t(0)e^{\frac{2h}{k}y},
$$
and as a result
$$
\sum_{i= M+1}^P{|x+{e_i}|^{-2t}}\rho_t\left(\frac{1}{x+{e_i}}\right)\le \rho_t(0)e^{\frac{2t}{k e_{M+1}}}\sum_{i= M+1}^P{{e_i}^{-2t}}.
$$
 Thus we obtain the following two-sided bound for $(L_t\rho_t)(x)$ in terms of $(L_{M,t}\rho_t)(x)$  for any $x\in[0,1]$
 \begin{equation}\label{eq: two sided bound}
 (L_{M,t}\rho_t)(x)\le (L_t\rho_t)(x)\le (L_{M,t}\rho_t)(x)+\rho_t(0)e^{\frac{2t}{k e_{M+1}}}\sum_{i= M+1}^P{{e_i}^{-2t}}.
 \end{equation}
After the truncation step (which is only needed if $P$ is very large), we approximate $\rho_t$ with continuous piecewise linear functions. To be precise, for any integer $N$ we partition the interval $[0,1]$ uniformly with $N+1$ points
\begin{equation}
\label{mesh}
0=x_1<x_2<\cdots<x_{N}<x_{N+1}=1,
\end{equation}
  where $x_i = (i+1)l$ with $l = 1/N$, then $\rho_t^I$ on each subinterval $[x_i,x_{i+1}]$ is of the form
\begin{equation}\label{eq:interpolant}
\rho_t^I(x) = \rho_t(x_{i})\frac{x_{i+1}-x}{l}+\rho_t(x_{i+1})\frac{x-x_i}{l},
\end{equation}
i.e. $\rho_t^I$ is piecewise linear interpolant of $\rho_t$. It is easy to see that at  any node
$\rho_t(x_i)=\rho_t^I(x_i)$, $i=1,2,\dots,N+1$. From standard approximation theory (see for example \cite[Thm.~2.1.4.1]{StoerJ_BulirschR_1993}) we know that for $x\in[x_i,x_{i+1}]$
$$
\rho_t^I(x)-\rho_t(x) = \frac{1}{2}(x_{i+1}-x)(x-x_i)\rho''_t(\xi),
$$
for some $\xi=\xi_x\in [x_i,x_{i+1}]$. Notice that from Theorem \ref{thm: Falk_Nuss 1} 
$$
\rho''_t(\xi)\le (2t)(2t+ 1)k^{-2}\rho_t(\xi).
$$
Thus, using the above estimate and the convexity of $\rho_t$ for any $y\in [x_j,x_{j+1}]$
we obtain
\begin{equation}\label{eq: estimate vI-v}
\begin{aligned}
0<\rho_t^I(y)-\rho_t(y)&\le \frac{1}{2}(x_{j+1}-x)(x-x_j) (2t)(2t+1)k^{-2}\rho_t(\xi_y)\\
&\le \frac{1}{2}(x_{j+1}-x)(x-x_j) (2t)(2t+1)k^{-2}e^{\frac{2tl}{k}}\rho^I_t(y),
\end{aligned}
\end{equation}
where in the last step we used the fact that $\rho_t(\xi_y)\le e^{\frac{2t}{k}|\xi-y|}\rho_t(y)\le e^{\frac{2tl}{k}} \rho^I_t(y)$. Now let 
$$
\err_{j,t} := \frac{1}{2}(x_{j+1}-x)(x-x_j) (2t)(2t+1)e^{\frac{2tl}{k}}k^{-2}.
$$
Provided $\err_{j,t}<1$ we obtain
\begin{equation}\label{eq: inequality for v both sides}
(1-\err_{j,t})\rho_t^I(y)<\rho_t(y)<\rho_t^I(y), \quad \mbox{ for all } y\in[0,1],
\end{equation}
and as a result for any $x\in[0,1]$
\begin{equation}\label{eq: two-sides Ls inequality}
\begin{split}
&\sum_{i\le  M}{|x+{e_i}|^{-2t}}(1-\err_{i,t})\rho^I_t\left(\frac{1}{x+{e_i}}\right)\\
&\quad\quad\quad\quad\quad \le (L_{M,t}\rho_t)(x)\le \sum_{i\le M}{|x+{e_i}|^{-2t}}\rho^I_t\left(\frac{1}{x+{e_i}}\right).
\end{split}
\end{equation}
Let $\vec{\alpha}_t$ be the $(N+1)$-vector with entries 
$$(\vec{\alpha}_t)_j=\rho_t(x_j)=\rho^I_t(x_j) \mbox{ for }j=1,\dots,N+1.$$ 
The next step consists of determining two $(N+1)\times (N+1)$ matrices $A_{M,t}$ and $B
_{M,t}$ such that
\begin{equation}\label{eq: matrices truncated}
\begin{aligned}
(A_{M,t}\vec{\alpha}_t)_j &= \sum_{i\le M}{|x_j+{e_i}|^{-2t}}(1-\err_{i,t})\rho^I_t\left(\frac{1}{x_j+{e_i}}\right)\\
(B_{M,t}\vec{\alpha}_t)_j &= \sum_{i\le M}{|x_j+{e_i}|^{-2t}}\rho^I_t\left(\frac{1}{x_j+{e_i}}\right)+\rho^I_t(0)e^{\frac{2}{k e_{M+1}}}\sum_{i= M+1}^P{{e_i}^{-2t}}.
\end{aligned}
\end{equation}
The main difficulty in constructing the above matrices is to locate the interval $[x_m,x_{m+1}]$ that contains $(x_j+{e_i})^{-1}$ for each $1\le i\le M$ and $1\le j\le N+1$.  
Once such interval with the corresponding index $m$ is located, then from the formula \eqref{eq:interpolant}, we obtain
$$
\rho^I_t\left(\frac{1}{x_j+{e_i}}\right)= (\vec{\alpha}_t)_m\frac{x_{m+1}-\frac{1}{x_j+{e_i}}}{l}+(\vec{\alpha}_t)_{m+1}\frac{\frac{1}{x_j+{e_i}}-x_m}{l},
$$
and as a result for these particular $i$ and $j$, $${|x_j+{e_i}|^{-2t}}(1-\err_{i,t})\frac{x_{m+1}-\frac{1}{x_j+{e_i}}}{l}\mbox{ and }{|x_j+{e_i}|^{-2t}}\frac{x_{m+1}-\frac{1}{x_j+{e_i}}}{l}$$ contribute to the $(j,m)$ entry of the matrices $A_{M,t}$ and $B_{M,t}$, respectively. Similarly  ${|x_j+{e_i}|^{-2t}}(1-\err_{i,t})\frac{\frac{1}{x_j+{e_i}}-x_m}{l}$  and ${|x_j+{e_i}|^{-2t}}\frac{\frac{1}{x_j+{e_i}}-x_m}{l}$ are the contributions to the $(j,m+1)$ entry of the matrices $A_{M,t}$ and $B_{M,t}$ respectively.
Since $\rho^I_t(0)=(\vec{\alpha}_t)_1$, we have that $e^{\frac{2}{k e_{M+1}}}\sum_{i= M+1}^P{{e_i}^{-2t}}$ is just a contribution to the $(j,1)$ entry of the matrix $B_{M,t}$. We would also like to point out that the matrices $A_{M,t}$ and $B_{M,t}$ are usually  very sparse and have many zero columns if $N$ is large. This feature can be exploited in computations.  

The next central result in the approach is \cite[Lemma 3.2]{FalkRS_NussbaumRD_2016b}, which we restate below.
\begin{lm}\label{lemma: monotonicity}
Let $M$ be an $N\times N$ matrix with non-negative entries  and $w$ an $N$-vector with strictly positive components. Then,
$$
\begin{aligned}
\text{if}\quad (Mw)_j &\geq \lambda w_j, \quad j=1,\dots,N,\ \text{then}\ r(M)\geq \lambda,\\
\text{if}\quad (Mw)_j &\leq \lambda w_j, \quad j=1,\dots,N,\ \text{then}\ r(M)\le \lambda.
\end{aligned}
$$
\end{lm}
Recall that by Theorem \ref{thm: Falk_Nuss 1} we have that 
$$(L_t\rho_t)(x_j)=r(L_t)\rho_t(x_j)$$ for any $j=1,\dots, N+1$, where $r(L_t)=\lambda_t$ denotes the spectral radius of $L_t$. 
Now notice that for $j=1,\dots, N+1,$
\begin{equation*}
\begin{split}
(A_{M,t}\vec{\alpha}_t)_j \overset{\eqref{eq: two-sides Ls inequality}}{\leq} (L_{M,t} \rho_t) (x_j) \overset{\eqref{eq: two sided bound}}{\leq} (L_t \rho_t) (x_j)=\lambda_t \rho_t (x_j)= r(L_t) (\vec{\alpha}_t)_j,
\end{split}
\end{equation*}
and
\begin{equation*}
\begin{split}
(B_{M,t}\vec{\alpha}_t)_j &\overset{\eqref{eq: two-sides Ls inequality}}{\geq} (L_{M,t} \rho_t) (x_j)+\rho^I_t(0)e^{\frac{2}{k e_{M+1}}}\sum_{i= M+1}^P{{e_i}^{-2t}} \\
&\overset{\eqref{eq: two sided bound}}{\geq} (L_t \rho_t) (x_j)=\lambda_t \rho_t (x_j)= r(L_t) (\vec{\alpha}_t)_j,
\end{split}
\end{equation*}
Therefore Lemma \ref{lemma: monotonicity} implies that for any $0<M \leq P$,
$$r(A_{M,t})\le r(L_t)=\lambda_h\le r(B_{M,t}).$$  
Let $h=\dim_{\cH}(J_E)$, and recall that according to Theorem \ref{thm: Falk_Nuss 1}, $r(L_{h})=\lambda_{h}=1$. Since the map $t\to \lambda_t$ is strictly decreasing if we find $\underline{h}$ such that 
$r(A_{M,{\underline{h}}})> 1$, then $r(L_{h})=1<r(A_{M,{\underline{h}}})\le r(L_{\underline{h}})$ and as a result $h> \underline{h}$. Similarly, if we find $\overline{h}$ such that 
$r(B_{M,{\overline{h}}})< 1$, then $r(L_{\overline{h}})\le r(B_{M,{\overline{h}}})< 1=r(L_{h})$ and as a result $h<\overline{h}$. In conclusion, we would have
  $\underline{h}< h< \overline{h}$, which is a rigorous effective estimate for the Hausdorff dimension of the set $J_E$. In the following Section we illustrate how this method can be used in order to estimate the Hausdorff dimension of the set $J_E$, for various concrete examples of infinite sets $E \subset \N$.  

\subsection{Concrete examples}

Based on the above description we have two strategies for approximating the Hausdorff dimension of infinite systems $\cf_{E}$, where $E=\{e_i\}_{i=1}^\infty$. 
We first select a sufficiently large $M \in \N$ and we let $F_M=\{e_i\}_{i=1}^M$. We also independently choose $N$, where recalling \ref{mesh} $N$ determines the size of the mesh. We then compute the two matrices $A_{M,t}$ and $B_{M,t}$, defined in \eqref{eq: matrices truncated}. Notice that in this approach the term $\rho^I_t(0)e^{\frac{2}{k e_{M+1}}}\sum_{i= M+1}^P{{e_i}^{-2t}}$ disappears. 
Using the method described in the previous section, we compute upper and lower bounds of the Hausdorff dimension of the limit set $J_{F_M}$, which we denote by $\underline{h}_M$ and $\overline{h}_M$, respectively. We then estimate the error $\dim_{\cH}(J_E)-\overline{h}_M$ using Theorem \ref{hein} as well as Propositions \ref{distco} and \ref{lyapbound}. 

The accuracy of $\underline{h}_M$ and $\overline{h}_M$ depends on the choice of $N$. Since we employ piecewise linear approximation it follows that for $N$ nodes the accuracy is essentially of order $N^{-2}$. The quantity $\dim_{\cH}(J_E)-\overline{h}_M$ obviously depends on $M$ but in this case it is harder to trace the exact dependance, as it is very much example dependent. 
We provide details of such estimates for each of our examples below. Finally we rigorously verify the accuracy of $\underline{h}_M$ and $\overline{h}_M$ using the interval arithmetic package \texttt{IntLab} \url{http://www.ti3.tu-harburg.de/rump/intlab/}.  

Although Theorem \ref{hein}  allows us to estimate the error $\dim_{\cH}(J_E)-\overline{h}_M$ for any $M \in \N$, for certain alphabets $E$ this error tends to zero rather slowly and one needs to consider quite large finite subsystems $F_M$ in order to obtain good estimates. As a result, in some cases these computations require substantial computer power. On the other hand this strategy is rather robust and the accuracy increases as $M$ and $N$ increase.

The second strategy allows us to obtain good accuracy with lighter computations. As in the previous case let $E=\{e_i\}_{i=1}^\infty$. Recall that Theorem \ref{t1j97} implies that  $\dim_{\cH} (J_{E})-\dim_{\cH} (J_{F_P})\to 0$ as $P\to \infty$, where  $F_P=\{e_i\}_{i=1}^P$, and we can use this fact to show that the Falk-Nussbaum method can be used to obtain lower and upper bounds for $\dim_{\cH}(J_E)$, not just for $\dim_{\cH}(J_{F_P})$. To be more precise let $\ve>0$ and let $P$ large enough such that  $\dim_{\cH} (J_{E})-\dim_{\cH} (J_{F_P})<\ve$, note that we only need the existence of such $P$ (justified by Theorem \ref{t1j97}) not its exact value. Let $M<P$ be some moderately large number and independently choose the mesh size $N$. We then compute the two matrices $A_{M,t}$ and $B_{M,t}$ from \eqref{eq: matrices truncated}. This time we use the full formula and the contribution to $\rho^I_t(0)$  can be estimated by
 \begin{equation}\label{eq: correction}
 \corr_{M}:=e^{\frac{2}{k e_{M+1}}}\sum_{i= M+1}^P{{e_i}^{-2t}}\le e^{\frac{2}{k e_{M+1}}}\sum_{i= M+1}^\infty{{e_i}^{-2t}},
 \end{equation}
which is  independent of $P$. Hence, by the method described in the previous section we obtain $ \underline{h}_{M},\overline{h}_{M}$ such that
$$\underline{h}_{M} \leq \dim_{\cH}(J_{F_P}) \leq \overline{h}_{M},$$
and consequently
$$\underline{h}_{M} \leq \dim_{\cH}(J_{E}) \leq \overline{h}_{M}+\ve.$$
This strategy has several advantages. First of all, there is no need to estimate $\dim_{\cH} (J_{E})-\dim_{\cH} (J_{F_M})$, and secondly $M$ can be taken rather moderately large. However, the main disadvantage is that it is not clear how $\overline{h}_{M}$ improves if $M$ and $N$ increase. Again, to rigorously verify the accuracy of $\underline{h}_{M}$ and $\overline{h}_{M}$ we use the interval arithmetic package \texttt{IntLab} \url{http://www.ti3.tu-harburg.de/rump/intlab/}.
  
At the end of this section we include Tables \ref{table: approach 1} and \ref{table: approach 2} which collect the dimension estimates for all our examples. We now provide some estimates of the quantities $\dim_{\cH}(J_E)-\overline{h}_M$ and $\corr_M$ for various examples of infinite alphabets $E \subset \N$.

\subsection{Odd integers.}
  
From the set $E_{odd} = \{1,3,5,\dots\}$ we select a finite subset $F_{odd}^M = \{1,3,5,\dots,2M+1\}$ for some large $M$.  Let 
$$\underline{h}_M \leq h_{F_{odd}^M} \leq \overline{h}_M$$ where, recalling Theorem \ref{t1j97}, $h_{F_{odd}^M}=\dim_{\cH}(J_{F_{odd}^M})$. Note that $\dim_{\cH}(J_{E_{odd}} )\geq \underline{h}_M$. Using the integral test we obtain
$$
\sum_{E_{odd} \stm F^M_{odd}} \, \|\f_e'\|^{h_M}_{\infty}\le \sum_{n\geq M+1}\frac{1}{(2n+1)^{2\underline{h}_M}}\le \int_{M}^\infty\frac{dx}{(2x+1)^{2\underline{h}_M}}=\frac{1}{2(2\underline{h}_M-1)(2M+1)^{2\underline{h}_M-1}}.
$$
From Proposition \ref{prop: Lyapunov} we have  $\chi_{E_{odd}}\geq  2\ln{\left(\frac{1+\sqrt{5}}{2}\right)}$ and since $1\in E_{odd}$ we have $K_{E_{odd}}=4$. As a result for an arbitrary $M$, Theorem \ref{hein} implies that
$$
\dim_{\cH} (J_{E_{odd}})-\dim_{\cH} (J_{F_{odd}^M})\le \frac{4^{\overline{h}_M}}{4\ln{\left(\frac{1+\sqrt{5}}{2}\right)}(2\underline{h}_M-1)(2M+1)^{2\underline{h}_F-1}}.
$$
Similarly, using the integral test, we also obtain from \eqref{eq: correction} (with $k=1$ and $h\le 1$),
$$
\corr_M = e^{\frac{2h}{2M+3}}\sum_{i\geq M+1}{{(2i+1)}^{-2h}}\le e^{\frac{2}{2M+3}}\int_M^\infty\frac{dx}{(2x+1)^{2h}}\le \frac{e^{\frac{1}{M}}}{2(2{h}-1)(2M+1)^{2{h}-1}}.
$$

\subsection{Even integers.}
This  case was previously considered in \cite{HeinemannSM_UrbanskiM_2002a}.  From the set $E_{even} = \{2,4,6,\dots\}$ we select a finite subset $F_{even}^M = \{2,4,6,\dots, 2M\}$ for some large $M$. Let 
$$\underline{h}_M \leq h_{F_{even}^M} \leq \overline{h}_M.$$
Using the integral test we obtain
$$
\sum_{E_{even} \stm F^M_{even}} \, \|\f_e'\|^{h_M}_{\infty}\le \sum_{n\geq M+1}\frac{1}{(2n)^{2\underline{h}_M}}\le \int_{M}^\infty\frac{dx}{(2x)^{2\underline{h}_M}}=\frac{1}{2(2\underline{h}_M-1)(2M)^{2\underline{h}_M-1}}.
$$
Propositions \ref{prop: distortion} and \ref{prop: Lyapunov} imply respectively that  $$K_{E_{even}}\le e^{2/3} \mbox{ and }\chi_{E_{even}}\geq 2 \ln{(1+\sqrt{2})}.$$  Therefore for an arbitrary $M$ we obtain the estimate
$$
\dim_{\cH} (J_{E_{even}})-\dim_{\cH} (J_{F})\le \frac{e^{2\overline{h}_M/3}}{4\ln{(1+\sqrt{2})}(2\underline{h}_M-1)(2M)^{2\underline{h}_M-1}}.
$$
Similarly, using the integral test, we also obtain from \eqref{eq: correction}  (with $k=2$ and $h\le 1$),
$$
\corr_M= e^{\frac{2h}{2(2M+2)}}\sum_{i\geq M+1}{{(2i)}^{-2h}}\le e^{\frac{1}{2M+2}}\int_M^\infty\frac{dx}{(2x)^{2h}}\le \frac{e^{\frac{1}{2M}}}{2(2{h}-1)(2M)^{2{h}-1}}.
$$

\subsection{Mod 3 integers.}
This example is very similar to the odd and even cases and  we omit some details. 
\subsubsection{1 Mod 3.}
Let $E_{1mod3}=\{1,4,7,\dots\}$. For $M \in \N$ let $F^M_{1mod3}=\{1,4,7,\dots, 3M+1\}$, and 
$$\underline{h}_M \leq h_{F^M_{1mod3}} \leq \overline{h}_M.$$
Propositions \ref{prop: distortion}  and \ref{prop: Lyapunov} imply that $\chi_{E_{1mod3}}\geq  2\ln{\left(\frac{1+\sqrt{5}}{2}\right)}$ , $K_{E_{1mod3}}=4$. Therefore,
$$
\sum_{E_{1mod3} \stm F^M_{1mod3}} \, \|\f_e'\|^{h_M}_{\infty}\le \frac{1}{3(2\underline{h}_M-1)(3M+1)^{2\underline{h}_M-1}}.
$$
and Theorem \ref{hein} implies that
$$
\dim_{\cH} (J_{E_{1mod3}})-\dim_{\cH} (J_{F^M_{1mod3}})\le \frac{4^{\overline{h}_M}}{6\ln{\left(\frac{1+\sqrt{5}}{2}\right)}(2\underline{h}_M-1)(3M+1)^{2\underline{h}_M-1}}.
$$
Similarly, using the integral test, we also obtain from \eqref{eq: correction}  (with $k=1$ and $h\le 1$),
$$
\corr_M = e^{\frac{2h}{3M+1}}\sum_{i\geq M+1}{{(3i+1)}^{-2h}}\le e^{\frac{2}{3M+1}}\int_M^\infty\frac{dx}{(3x+1)^{2h}}\le \frac{e^{\frac{2}{3M}}}{3(2{h}-1)(3M+1)^{2{h}-1}}.
$$

\subsubsection{ 2 Mod 3.}
Let $E_{2mod3}=\{2,5,8,\dots\}$. For $M \in \N$ let  $F^M_{2mod3}=\{2,5,8,\dots, 3M+2\}$ and 
$$\underline{h}_M \leq h_{F^M_{2mod3}} \leq \overline{h}_M.$$
Propositions \ref{prop: distortion}  and \ref{prop: Lyapunov} imply that $\chi_{E_{2mod3}}\geq  2\ln{(1+\sqrt{2})}$ , $K_{E_{2mod3}} \le e^{2/3}$. Therefore
$$
\sum_{E_{2mod3} \stm F^M_{2mod3}} \, \|\f_e'\|^{h_M}_{\infty}\le \frac{1}{3(2\underline{h}_M-1)(3M+2)^{2\underline{h}_M-1}},
$$
and Theorem \ref{hein} implies that
$$
\dim_{\cH} (J_{E_{2mod3}})-\dim_{\cH} (J_{F^M_{2mod3}})\le \frac{e^{2\overline{h}_M/3}}{6\ln{(1+\sqrt{2})}(2\underline{h}_M-1)(3M+2)^{2\underline{h}_M-1}}.
$$
Similarly, using the integral test, we also obtain from \eqref{eq: correction}  (with $k=2$ and $h\le 1$),
$$
\corr_M= e^{\frac{h}{3M+2}}\sum_{i\geq M+1}{{(3i+2)}^{-2h}}\le e^{\frac{1}{3M+2}}\int_M^\infty\frac{dx}{(3x+2)^{2h}}\le \frac{e^{\frac{1}{3M}}}{3(2{h}-1)(3M+2)^{2{h}-1}}.
$$
\subsubsection{ 0 Mod 3.}
Let $E_{0mod3}=\{3,6,9,\dots\}$.  For $M \in \N$ let $F^M_{0mod3}=\{3,6,9,\dots, 3M+3\}$ and 
$$\underline{h}_M \leq h_{F^M_{0mod3}} \leq \overline{h}_M.$$
Propositions \ref{prop: distortion}  and \ref{prop: Lyapunov} imply that
$\chi_{E_{0mod3}}\geq  2\ln{\left(\frac{3+\sqrt{13}}{2}\right)}$ , $K_{E_{0mod3}}\le e^{1/4}$. Therefore
$$
\sum_{E_{0mod3} \stm F^M_{0mod3}} \, \|\f_e'\|^{h_M}_{\infty}\le \frac{1}{3(2\underline{h}_M-1)(3M+3)^{2\underline{h}_M-1}}.
$$
and Theorem \ref{hein} implies that
$$
\dim_{\cH} (J_{E_{0mod3}})-\dim_{\cH} (J_{F^M_{0mod3}})\le \frac{e^{\overline{h}_M/4}}{6\ln{\left(\frac{3+\sqrt{13}}{2}\right)}(2\underline{h}_M-1)(3M+3)^{2\underline{h}_M-1}}.
$$
Similarly, using the integral test, we also obtain from \eqref{eq: correction}  (with $k=3$ and $h\le 1$),
$$
\corr_M = e^{\frac{2h}{3(3M+3)}}\sum_{i\geq M+1}{{(3i+3)}^{-2h}}\le e^{\frac{2}{9M}}\int_M^\infty\frac{dx}{(3x+3)^{2h}}\le \frac{e^{\frac{2}{9M}}}{3(2{h}-1)(3M+3)^{2{h}-1}}.
$$

\subsection{Primes.}

We denote $E_{prime} = \{2,3,5,\dots\}=\{p(1),p(2),p(3),\dots\}$, where $p(n)$ is the prime function that gives you the n-th prime. We select a finite subset $F_{prime}^M = \{p(1),p(2),p(3),\dots,p(M-1)\}$ for some large $M$. Let 
$$\underline{h}_M \leq h_{F_{prime}^M} \leq \overline{h}_M.$$ It follows from  \cite[Lemma 1 and Theorem 3]{DusartP_1999} that for $n\geq 6$,
$$
n(\ln{n}+\ln{\ln{n}}-1)\le p(n) \le n(\ln{n}+\ln{\ln{n}}).
$$
Using this estimate and  the integral test we obtain
$$
\begin{aligned}
\sum_{E_{prime} \stm F^M_{prime}} \, \|\f_e'\|^{h_M}_{\infty}&\le \sum_{n\geq M}\frac{1}{p(n)^{2\underline{h}_M}}\le \sum_{n\geq M} \frac{1}{n^{2\underline{h}_M}(\ln{n}+\ln{\ln{n}}-1)^{2\underline{h}_M}}\\
& \le \frac{1}{(\ln{M}+\ln{\ln{M}}-1)^{2\underline{h}_M}}\sum_{n\geq M} \frac{1}{n^{2\underline{h}_M}}\\
& \le \frac{1}{(\ln{M}+\ln{\ln{M}}-1)^{2\underline{h}_M}}\int_{M-1}^\infty \frac{dx}{x^{2\underline{h}_M}}\\
& = \frac{1}{(\ln{M}+\ln{\ln{M}}-1)^{2\underline{h}_M}(2\underline{h}_M-1)(M-1)^{2\underline{h}_M-1}}.
\end{aligned}
$$
From Propositions \ref{prop: distortion} and \ref{prop: Lyapunov} it follows respectively that 
$$\chi_{E_{prime}}\geq  2\ln{(1+\sqrt{2})} \mbox{ and  }K_{E_{prime}}\le e^{2/3}.$$ Hence Theorem \ref{hein} implies that for an arbitrary $M$,
\begin{equation*}
\begin{split}
\dim_{\cH} (J_{E_{prime}})&-\dim_{\cH} (J_{F^M_{prime}}) \\
&\le \frac{e^{2\overline{h}_M/3}}{2\ln{(1+\sqrt{2})}{(\ln{M}+\ln{\ln{M}}-1)^{2\underline{h}_M}(2\underline{h}_M-1)(M-1)^{2\underline{h}_M-1}}}.
\end{split}
\end{equation*}
Similarly, using the integral test, we also obtain from \eqref{eq: correction}  (with $k=2$ and $h\le 1$),
$$
\begin{aligned}
\corr_M= e^{\frac{2h}{k p(M)}}\sum_{i\geq M}p(i)^{-2h}&\le \frac{e^{\frac{1}{M\ln{M}}}}{(\ln{M}+\ln{\ln{M}}-1)^{2\underline{h}_M}}\int_{M-1}^\infty\frac{dx}{x^{2h}}\\
&\le \frac{e^{\frac{1}{M\ln{M}}}}{(\ln{M}+\ln{\ln{M}}-1)^{2\underline{h}_M}(2h-1)(M-1)^{2h-1}}.
\end{aligned}
$$

\subsection{Squares}
From the set $E_{square} = \{1,4,9,\dots\}$ we select a finite subset $F_{square}^M = \{1,4,9,\dots, M^2\}$ for some large $M$.  Let 
$$\underline{h}_M \leq h_{F_{square}^M} \leq \overline{h}_M.$$ Using the integral test we obtain
$$
\sum_{E_{square} \stm F^M_{square}} \, \| \f_e'\|^{h_M}_{\infty}\le \sum_{n\geq M+1}\frac{1}{n^{4\underline{h}_M}}\le \int_{M}^\infty\frac{dx}{x^{4\underline{h}_M}}=\frac{1}{(4\underline{h}_M-1)M^{4\underline{h}_M-1}}.
$$
Propositions \ref{prop: distortion} and \ref{prop: Lyapunov} imply respectively that $$K_{E_{square}}=4\mbox{ and  }\chi_{\tilde{\mu}_h}\geq 2\ln{\left(\frac{1+\sqrt{5}}{2}\right)}.$$As a result for an arbitrary $M$ we obtain the estimate
$$
\dim_{\cH} (J_{E_{square}})-\dim_{\cH} (J_{F_{square}^M})\le \frac{4^{\overline{h}_M}}{2\ln{\left(\frac{1+\sqrt{5}}{2}\right)}(4\underline{h}_M-1)M^{4\underline{h}_M-1}}.
$$
Similarly, using the integral test, we also obtain from \eqref{eq: correction}  (with $k=1$ and $h\le 1$),
$$
\corr_M= e^{\frac{2h}{(M+1)^2}}\sum_{i\geq M+1}{{i}^{-4h}}\le e^{\frac{2}{M^2}}\int_M^\infty\frac{dx}{x^{4h}}\le \frac{e^{\frac{2}{M^2}}}{(4{h}-1)M^{4{h}-1}}.
$$

\subsection{Powers.}
Now we consider the case when $E$ consists of the powers of $m_0$, for some $m_0\in \N$.  If $1 \in E$, i.e. if the zero power is included, we let $E_{power+}=\{1,m_0,m_0^2, m_0^3, \dots\}$ and we select a finite subset $F_{power+}^M = \{1,m_0,m_0^2,\dots,m_0^M\}$ for some $M$.  Let 
$$\underline{h}_M \leq h_{F_{power+}^M} \leq \overline{h}_M.$$ By the properties of the geometric sums 
$$
\sum_{E_{power+} \stm F_{power+}} \, \|\f_e'\|^{h_M}_{\infty}\le \sum_{n\geq M+1}\frac{1}{m_0^{2n\underline{h}_M}}=\frac{1}{m_0^{2M\underline{h}_M}-1},
$$
 As in the case of odd numbers we have $\chi_{E_{power+}}\geq   2\ln{\left(\frac{1+\sqrt{5}}{2}\right)}$ and as a result in this case
$$
\dim_{\cH} (J_{E_{power+}})-\dim_{\cH} (J_{F_{power+}^M})\le \frac{4^{\overline{h}_M}}{2\ln{\left(\frac{1+\sqrt{5}}{2}\right)}\left({m_0^{2M\underline{h}_M}-1}\right)}.
$$

If $1$ is not included then we denote $E_{power+}=\{m_0,m_0^2, m_0^3, \dots\}$ and we select a finite subset $F_{power}^M = \{m_0,m_0^2,\dots,m_0^M\}$ for some $M$.  Propositions \ref{prop: distortion} and \ref{prop: Lyapunov} imply respectively that $$K_{E_{power}}\le  \exp \left( \frac{2}{m_0^2-1}\right) \mbox{ and } \chi_{E_{power}}\geq  2 \ln \left( \frac{m_0+\sqrt{m_0^2+4}}{2} \right).$$ Therefore 
$$
\dim_{\cH} (J_{E_{power}})-\dim_{\cH} (J_{F_{power}^M})\le \frac{\exp \left( \frac{2\overline{h}_M}{m_0^2-1}\right)}{2 \ln \left( \frac{m_0+\sqrt{m_0^2+4}}{2} \right)\left({m_0^{2M\underline{h}_M}-1}\right)}.
$$
In both cases the error is very negligible even for moderate $M$ and the quality of the approximation depends only one the good approximation of $\underline{h}_M$ and $\overline{h}_M$, i.e. on $N$ which is the number of  nodes in $[0,1]$.
Similarly, using the properties of geometric series, we also obtain from \eqref{eq: correction}  (with $k=1$ or $k=m0$ and $h\le 1$),
$$
\corr_M= e^{\frac{2h}{k m_0^M}}\sum_{i\geq M+1}\frac{1}{m_0^{2i{h}}}\le \frac{e^{\frac{2}{m_0^M}}}{m_0^{2M{h}}-1}.
$$

\subsection{Lacunary sequences}
We now consider the lacunary sequence, $E_{lac}=\{2^{n^2}\}_{n=1}^\infty$. We select a finite subset $F_{lac}^M = \{2,2^{2^2},2^{3^2},\dots,2^{M^2}\}$ for some $M$ and we let
$$\underline{h}_M \leq h_{F_{lac}^M} \leq \overline{h}_M.$$
Since the decay is very fast we only use a crude estimate based on elementary properties of geometric series
$$
\sum_{E_{lac} \stm F^M_{lac}} \, \|\f_e'\|^{h_M}_{\infty}\le \sum_{n\geq M+1}2^{-n^2\underline{h}_M}=2^{-(M+1)^2\underline{h}_M} \sum_{n=0}^\infty 2^{-n^2\underline{h}_M}\le 2^{-(M+1)^2\underline{h}_M}\frac{2^{\underline{h}_F}}{2^{\underline{h}_M}-1}.
$$
Propositions \ref{prop: distortion} and \ref{prop: Lyapunov} imply respectively that $K_{E_{lac}}\le e^{2/3}$ and $\chi_{E_{lac}}\geq 2\ln{(1+\sqrt{2})}$, therefore 
$$
\dim_{\cH} (J_{E_{lac}})-\dim_{\cH} (J_{F^M_{lac}})\le \left(\frac{2^{\underline{h}_M}}{2^{\underline{h}_M}-1}\right)\frac{e^{\frac{2\overline{h}_M}{3}}}{2^{(M+1)^2\underline{h}_M}2\ln{(1+\sqrt{2})}}.
$$
Note that the error, i.e. the right hand side, is very small even for moderately large $M$.
Similarly from \eqref{eq: correction} we have that $
\corr_M\le 2^{-M^2{h}} .
$


\begin{table}[ht]
\caption{Hausdorff dimension estimates using strategy 1}
\centering
\vspace{.1in}
\begin{tabular}{|l|c|c|c|c|}
\hline
Subsystem           & Hausdorff dim. interval & accuracy & M & N \\
\hline
$E_{odd}=\{1,3,5,\dots\}$ & $[0.821160, 0.821223]$  & $6.3e-5$ & $1e+7$ & $100$ \\
\hline
$E_{even}=\{2,4,6,\dots\}$ & $[0.71936, 0.72001]$ & $6.5e-4$ & $1e+7$ & $100$ \\
\hline
$E_{1mod3}=\{1,4,7,\dots\}$ & $[0.74352, 0.74398]$ & $4.6e-4$ & $1e+7$ & $100$ \\
\hline
$E_{2mod3}=\{2,5,8,\dots\}$ & $[0.66490, 0.66795]$ & $3.1e-3$ & $1e+7$ & $100$ \\
\hline
$E_{0mod3}=\{3,6,9,\dots\}$ & $[0.63956, 0.64916]$ & $9.6e-3$  & $1e+7$ & $100$ \\
\hline
$E_{prime}=\{2,3,5,\dots\}$ & $[0.67507, 0.67519]$ & $1.2e-4 $ & $2.6e+7$ & $100$ \\
\hline
$E_{square}=\{1,2^2,3^2,\dots\}$ & $[0.59825568, 0.59825603]$ & $3.5e-7 $ & $1e+5$ & $1000$ \\
\hline
$E_{power2+}=\{1,2,2^2,\dots\}$ & $[0.7339041186, 0.7339041234]$ & $ 4.8e-9$ & $60$ & $6000$ \\
\hline
$E_{power2}=\{2,2^2,2^3,\dots\}$ & $[0.4720715327, 0.4720715331]$ & $ 4.0e-10$ & $60$ & $6000$ \\
\hline
$E_{power3+}=\{1,3,3^2,\dots\}$ & $[0.5627284510, 0.5627284539]$ & $ 2.9e-9$ & $50$ & $6000$ \\
\hline
$E_{power3}=\{3,3^2,3^3,\dots\}$ & $[0.3105296859, 0.3105296860]$ & $ 1.0e-10$  & $50$ & $6000$ \\
\hline
$E_{lac}=\{2,2^{2^2},2^{3^2},\dots\}$ & $[0.2362689121, 0.2362689123]$ & $ 2.0e-10$  & $12$ & $5000$ \\
\hline
\end{tabular}
\label{table: approach 1}
\end{table}

\begin{table}[ht]
\caption{Hausdorff dimension estimates using strategy 2}
\centering
\vspace{.1in}
\begin{tabular}{|l|c|c|c|c|}
\hline
Subsystem           & Hausdorff dim. interval & accuracy & M & N \\
\hline
$E_{odd}=\{1,3,5,\dots\}$ & $[0.821143, 0.821177]$  & $3.4e-05 $ & $1e+6$ & $200$ \\
\hline
$E_{even}=\{2,4,6,\dots\}$ & $[0.719109, 0.719498]$ & $3.9e-04$ & $1e+6$ & $200$ \\
\hline
$E_{1mod3}=\{1,4,7,\dots\}$ & $[0.743404, 0.743586]$ & $1.8e-04$ & $1e+6$ & $200$ \\
\hline
$E_{2mod3}=\{2,5,8,\dots\}$ & $[0.664488, 0.665462]$ & $9.7e-04$ & $2e+6$ & $200$ \\
\hline
$E_{0mod3}=\{3,6,9,\dots\}$ & $[0.638856, 0.640725]$ & $1.9e-03$  & $2e+6$ & $200$ \\
\hline
$E_{prime}=\{2,3,5,\dots\}$ & $[0.675044, 0.675228]$ & $1.8e-04$ & $5.7e+6$ & $200$ \\
\hline 
$E_{square}=\{1,4,9,\dots\}$ & $[0.59825575,0.59825579]$ & $4.0e-08 $ & $1e+5$ & $5000$ \\
\hline
$E_{power2+}=\{1,2,4,\dots\}$ & $[0.73390397, 0.73390415]$ & $ 1.8e-07 $ & $50$ & $1000$ \\
\hline
$E_{power2}=\{2,4,8,\dots\}$ & $[0.472071525, 0.472071536]$ & $1.1e-08$ & $50$ & $1000$ \\
\hline
$E_{power3+}=\{1,3,9,\dots\}$ & $[0.56272836, 0.56272847]$ & $ 1.1e-07$ & $40$ & $1000$ \\
\hline
$E_{power3}=\{3,9,27,\dots\}$ & $[0.310529684, 0.310529686]$ & $ 2.0e-09$  & $40$ & $1000$ \\
\hline
$E_{lac}=\{2,2^{2^2},2^{3^2},\dots\}$ & $[0.236268909, 0.236268937]$ & $ 2.8e-08 $  & $10$ & $1000$ \\
\hline
\end{tabular}
\label{table: approach 2}
\end{table}

\section{Dimension Spectrum of Continued Fractions Subsystems}
\label{sec:spec}
In this section we are going to study the dimension spectrum of continued fractions generated by various infinite subsets of the natural numbers. We start by presenting several general criteria related to the dimension spectrum of continued fractions.

\subsection{Dimension spectrum criteria for continued fractions}
For $r  \geq 0$ we consider the sequences
$$\a_n(r)= \left( \frac{1}{n+1}\right)^{2r} \mbox{ and } \beta_n(r)=\left( \frac{2}{n+2}\right)^{2r}, \quad \mbox{ for }n \in \N.$$
\begin{propo} 
\label{compsums}
Let $e_n : \N \ra \N$ be strictly increasing and let $t \geq s > 0$. If for some $k \in \N$
$$\sum_{n=k+1}^\infty \a_{e_n} (t) \geq \beta_{e_k}(t), \quad \text{then}\quad \sum_{n=k+1}^\infty \a_{e_n} (s) \geq \beta_{e_k}(s).$$
\end{propo}
\begin{proof} For simplicity let $\a_{e_n}:=\a_{e_n}(1/2)$ and $\beta_{e_n}:=\beta_{e_n}(1/2)$. If $n>k$ then $e_n>e_k$ and
$$\frac{\a_{e_n}}{\beta_{e_k}}=\frac{e_k+2}{2(e_n+1)}<1.$$
Hence the functions
$$\f_n(r)= \frac{\a_{e_n}(r)}{\beta_{e_k}(r)}=\left( \frac{\a_{e_n}}{\beta_{e_k}}\right)^{2r}$$
are decreasing for all $n \geq k+1$ and consequently the function
$$\f(r)=\sum_{n=k+1}^\infty \frac{ \a_{e_n} (r)}{\beta_{e_k}(r)}$$ is also decreasing. Therefore if $t \geq s>0$
$$\sum_{n=k+1}^\infty \frac{ \a_{e_n} (s)}{ \beta_{e_k}(s)} \geq \sum_{n=k+1}^\infty \frac{ \a_{e_n} (t)}{\beta_{e_k}(t)} \geq 1.$$
The proof is complete.
\end{proof}
Combining \cite[Lemma 4.3]{KessebohmerM_ZhuS_2006}, \cite[Theorem 6.4]{ChousionisV_LeykekhmanD_UrbanskiM_2018a} and Proposition \ref{compsums} we have the following.
\begin{propo}
\label{keypropo}
Let $E=\{e_n\}_{n \in \N}$ be an increasing sequence of natural numbers and denote by $I(m)= \{e_i\}_{i=1}^m, m \in \N,$ the initial segments of $E$. If there exist $t,s$ such that $0 \leq t \leq s \leq \dim_{\cH}(\cf_E)$ and some $m \in \N$ such that
\begin{enumerate}[label=(\roman*)]
\item $P_{I(m)}(t) \leq 0$,
\item \label{sumcond}$\sum_{n=k+1}^\infty \a_{e_n}(s) \geq \beta_{e_k}(s)$ for all $k \geq m+1$,
\end{enumerate}
then $[t,s] \subset DS (\cf_{E})$.
 \end{propo}
We state as a corollary a special case of the previous proposition.
\begin{corollary}
\label{keycoro}
Let $E=\{e_n\}_{n \in \N}$ be an increasing sequence of natural numbers. If there exists $s \geq 0$ such that
$$\sum_{n=k+1}^\infty \a_{e_n}(s) \geq \beta_{e_k}(s)$$ for all $k \geq 2$,
then $[0,\min\{s, \dim_{\cH}(J_E)\}] \subset DS (\cf_{E})$.
 \end{corollary}
We are also going to use the following criterion which follows from  \cite[Proposition 6.17]{ChousionisV_LeykekhmanD_UrbanskiM_2018a}. Recall that if $E \subset \N$, we denote best distortion constant of $\cf_E$ by $K_E:=K_{\cf_E}$.
\begin{propo}
\label{keypropo2}
Let $E=\{e_n\}_{n \in \N}$ be an increasing sequence of natural numbers and denote by $I(m)= \{e_i\}_{i=1}^m, m \in \N,$ the initial segments of $E$. If there exist $t,s$ such that $0 \leq t \leq s \leq \dim_{\cH}(\cf_E)$ and some $m \in \N$ such that
\begin{enumerate}[label=(\roman*)]
\item $P_{I(m)}(t) \leq 0$,
\item \label{sumcond}$\sum_{n=k+1}^\infty \| \f'_{e_n}\|^{s}_\infty \geq K_E^{2s} \| \f'_{e_k}\|^{s}_\infty$ for all $k \geq m+1$,
\end{enumerate}
then $[t,s] \subset DS (\cf_{E})$.
\end{propo}
We state as a corollary a special case of the previous proposition.
\begin{corollary}
\label{keycoro2}
Let $E=\{e_n\}_{n \in \N}$ be an increasing sequence of natural numbers. If there exists $s \geq 0$ such that
$$\sum_{n=k+1}^\infty \| \f'_{e_n}\|^{s}_\infty \geq K_E^{2s} \| \f'_{e_k}\|^{s}_\infty$$ for all $k \geq 2$,
then $[0,\min\{s, \dim_{\cH}(J_E)\}] \subset DS (\cf_{E})$.
 \end{corollary}
 
 \begin{remark} One may notice that Proposition \ref{keypropo} and \ref{keypropo2} are very similar. It turns out that Proposition \ref{keypropo} is more effective when $1 \in E$, in all other cases Proposition \ref{keypropo2} is preferable. This is because $K_E$ decreases as $\min E$ increases.
 \end{remark}
 The following theorem provides a checkable condition for determining if subsets of the dimension spectrum are nowhere dense. It follows from results in \cite{KessebohmerM_ZhuS_2006} and \cite{ChousionisV_LeykekhmanD_UrbanskiM_2018a}. We provide the details below.
 \begin{thm}
 \label{kessenode}
Let $E=\{e_n\}_{n \in \N}$ be an increasing sequence of natural numbers such that the system $\cS=\{\f_e\}_{e \in E}$ is absolutely regular. If there exist $0<s<r \leq \dim_{\cH}(J_E)$ such that
\begin{equation}
\label{sumndcond}
\sum_{n=k+1}^\infty \| \f'_{e_n}\|^{t}_\infty < K_E^{-2t} \| \f'_{e_k}\|^{t}_\infty\mbox{ for all }k \in \N, t\in (s,r),
\end{equation}
then $DS(\cS) \cap [s,r]$ is nowhere dense.
\end{thm}
\begin{proof}  If $F \subset E$ is nonempty and finite we let $N(F) \in \N$ such that $\max(F)=e_{N(F)}$. Moreover we denote
$$F^-=F \stm \max (F),$$
and
$$F_\infty^+=(F \stm \max (F)) \cup \{e_{N(F)+1}, e_{N(F)+2}, e_{N(F)+3}, \dots\}.$$
According to \cite[Theorem 2.4]{KessebohmerM_ZhuS_2006} it suffices to check that for every nonempty finite set $F \subset E$ and every $t \in [s,r],$
\begin{equation}
\label{kessend}
e^{P_{F_\infty^+}(t)}<e^{P_{F}(t)}.
\end{equation}  
By \cite[Proposition 4.7]{ChousionisV_LeykekhmanD_UrbanskiM_2018a} we have that
\begin{equation*}
e^{P_{F_\infty^+}(t)} \leq e^{P_{F^-}(t)}+K_E^t \sum_{n=N(F)+1} \|\f'_{e_n}\|_\infty^t 
\end{equation*}
and 
\begin{equation*}e^{P_F(t)} \geq e^{P_{F^-}(t)}+K_E^{-t} \|\f'_{\max(F)}\|^t_{\infty}.
\end{equation*}
Therefore
\begin{equation}
\label{pressest}
e^{P_{F_\infty^+}(t)} \leq e^{P_F(t)}+K_E^t \sum_{n=N(F)+1} \|\f'_{e_n}\|_\infty^t -K_E^{-t} \|\f'_{\max(F)}\|^t_{\infty}.
\end{equation} 
Finally \eqref{kessend} follows by \eqref{pressest} and \eqref{sumndcond}. The proof is complete.
\end{proof}
We are now ready to prove  our main theorems, establishing fullness of spectrum for systems associated to various well known subsets of the natural numbers.
\subsection{Arithmetic progressions.}
Our ultimate goal in this section is to establish that real continued fractions with entries belonging to any arithmetic progression have full spectrum. The proof has several steps and we will start with the most well known examples of arithmetic progressions.
 \begin{thm}
 \label{specsubsyst}
The systems $\cf_{{odd}}$ and  $\cf_{{even}}$ have full spectrum.
 \end{thm}
 \begin{proof} We first consider $\cf_{{even}}$. Recall that as we proved in Section \ref{sec:dim}, $\dim_{\cH}(J_{E_{even}}) < 0.72:=s$. We aim to apply Corollary \ref{keycoro2} for $E=2 \N$, in particular we would like to verify that \begin{equation}
\label{evensum}
\sum_{n=k+1}^\infty \left( \frac{1}{2n}\right)^{2s} \geq K_{2 \N} \left( \frac{1}{2k}\right)^{2s}
\end{equation}
holds for all $k \geq 2$. By the integral test,
$$\sum_{n=k+1}^\infty \left( \frac{1}{n}\right)^{2s}  \geq \int_{k+1}^\infty (x)^{-2s} dx=  \frac{1}{2s-1}(k+1)^{1-2s}.$$
Hence if $k \in \N$ satisfies
\begin{equation}
\label{even2}
\frac{k^{2s}}{(k+1)^{2s-1}} \geq \exp(2/3)(2s-1),
\end{equation}
it also satisfies \eqref{evensum}. One can readily verify that \eqref{even2} holds for $k=2$. Moreover it is easy to check that the function
$$h_e(x)=\frac{x^{2s}}{(x+1)^{2s-1}}$$
is increasing for $x \geq 0$. Therefore \eqref{even2} is satisfied for all $k \geq 2$ and Corollary \ref{keycoro2} implies that $\cf_{{even}}$ has full dimension spectrum.

We now move on to $\cf_{{odd}}$. Recall from Section \ref{sec:dim} that $\dim_{\cH}(J_{E_{odd}}) < 0.822:=s$. We will apply Corollary \ref{keycoro} for $E=2 \N-1$.  If $e_n=2n+1$, note that it suffices to verify that 
\begin{equation}
\label{oddsum}\sum_{n=k+1}^\infty \left( \frac{1}{n+1}\right)^{2s} \geq 16^s \left( \frac{1}{2k+3}\right)^{2s}
\end{equation}
for all $k \geq 1$.
By the integral test,
$$\sum_{n=k+1}^\infty \left( \frac{1}{n+1}\right)^{2s}  \geq \int_{k+1}^\infty (x+1)^{-2s} dx=\frac{(k+2)^{1-2s}}{2s-1}.$$
Hence if $k \in \N$ satisfies
\begin{equation}
\label{odd2}
\frac{(2k+3)^{2s}}{(k+2)^{2s-1}} \geq 16^s(2s-1),
\end{equation}
then it also satisfies \eqref{oddsum}. Indeed \eqref{odd2} holds for $k=1$ and it is easy to see that the function 
$$h_o(x)=\frac{(2x+3)^{2s}}{(x+2)^{2s-1}}$$
is increasing for $x \geq 0$. Therefore \eqref{odd2} and  (consequently) \eqref{oddsum} hold for all $k \geq 1$. Therefore  Corollary \ref{keycoro} implies that $\cf_{2\N-1}$ has full dimension spectrum. The proof is complete.
\end{proof}

We will now prove a special case of Theorem \ref{progrespc}; the case of arithmetic progression whose first element is less or equal than the step of the progression.

\begin{thm} 
\label{arithm1}Let $q,s \in \N$  such that $s \leq q$ and consider the arithmetic progressions $E_{s,q}=\{s+n q\}_{n \geq 0}$. Then $\cf_{E_{s,q}}$ has full spectrum.
\end{thm}
\begin{proof}
For $q=1$ the result is due to Kesseb\"ohmer and Zhu \cite{KessebohmerM_ZhuS_2006} who gave a positive answer to the Texan conjecture. For $q=2$ the result is contained in Theorem \ref{specsubsyst}. Hence we can assume that $q \geq 3$. 

Recall from Section \ref{sec:dim} that $\dim_{\cH}(J_{E_{1,3}}) < 0.744:=s_0$. Therefore if $q \geq 3$ and $s \leq q$, Proposition \ref{bijpropo} implies that
\begin{equation}
\label{dimeqbd2}\dim_{\cH}(J_{E_{s,q}}) \leq  \dim_{\cH}(J_{E_{1,3}}) < s_0.
\end{equation}
We aim to apply Proposition \ref{keypropo2} and a bootstrapping argument as in the proof of \cite[Theorem 1.4]{ChousionisV_LeykekhmanD_UrbanskiM_2018a}. We will first determine for which $k \geq 1$,
\begin{equation}
\label{qsum2}\sum_{n=k+1}^\infty \|\f'_{e_n}\|_\infty^{h_{s,q}} \geq 4^{2{h_{s,q}}} \|\f'_{e_k}\|_\infty^{h_{s,q}}
\end{equation}
where  $h_{s,q}=\dim_{\cH}(J_{E_{s,q}})$ and $e_k=s+kq$.
By \cite[Lemma 6.17]{ChousionisV_LeykekhmanD_UrbanskiM_2018a} and \eqref{dimeqbd2}, the condition \eqref{qsum2} holds as long as
\begin{equation}
\label{qsum3}\sum_{n=k+1}^\infty (nq+s)^{-2s_0} \geq 4^{2s_0} (kq+s)^{-2s_0}.
\end{equation}
We have that
$$\sum_{n=k+1}^\infty (nq+s)^{-2s_0} \geq \sum_{n=k+1}^\infty ((n+1)q)^{-2s_0} \geq q^{-2s_0}\int_{k+1}^\infty (x+1)^{-2s_0}\,dx=q^{-2s_0} \frac{(k+2)^{1-2s_0}}{2s_0-1}.$$
Therefore \eqref{qsum2} holds if 
\begin{equation*}
q^{-2s_0} \frac{(k+2)^{1-2s_0}}{2s_0-1} \geq 4^{2s_0} (kq)^{-2s_0},
\end{equation*}
or equivalently if
\begin{equation}
\label{eqinterm2}\frac{k^{2 s_0}}{(k+2)^{2 s_0-1}} \geq 4^{2s_0} (2s_0-1).
\end{equation}
Note that the function
$$h(x)=\frac{x^{2 s_0}}{(x+2)^{2 s_0-1}}$$
is increasing for $x>0$. Therefore, since we easily check by subsitution that \eqref{eqinterm2} holds for $k=5$, we deduce that \eqref{qsum2} holds for all $k \geq 5$.

Since $q \geq 3$ it follows by Proposition \ref{bijpropo} that for all $m \in \N$
\begin{equation}
\label{initicomp}
\dim_{\cH}(J_{I_{s,q}(m)}) \leq \dim_{\cH}(J_{I_{1,3}(m)}).
\end{equation}
where $I_{s,q}(m)=\{s+nq: n=0, \dots, m-1\}$. 

Using the code of Falk-Nussbaum from \cite{FalkRS_NussbaumRD_2016a}, and verifying our estimate using \texttt{IntLab} as in Section \ref{sec:dim}, we deduce that
$$\dim_{\cH}(J_{I_{1,3}(5)}) \in [0.597742, 0.597746].$$
We let $s_1=0.597746$ and we distinguish two cases. 

First assume that $s_1\geq \dim_{\cH}(J_{E_{s,q}})$. Arguing as before we deduce that
\begin{equation}
\label{eqinterm3}\frac{k^{2 s_1}}{(k+2)^{2 s_1-1}} \geq 4^{2s_1} (2s_1-1)
\end{equation}
for all $k \geq 2$. Note that \eqref{eqinterm3} implies that 
\begin{equation}
\label{qsum4}\sum_{n=k+1}^\infty \|\f'_{e_n}\|_\infty^{s_1} \geq 4^{2s_1} \|\f'_{e_k}\|_\infty^{s_1}
\end{equation}
holds for all $k \geq 2$. Using the Falk-Nussbaum code once more (as well as \texttt{IntLab} verification)  we deduce that
$$\dim_{\cH}(J_{I_{1,3}(2)}) \in [0.411181, 0.411184].$$
Hence if $s_2=0.411185$ 
\begin{equation}
\label{press2}
P_{I_{s,q}(2)}(s_2) <0.
\end{equation} 
By \eqref{qsum4}, \eqref{press2} and Proposition \ref{keypropo2} we deduce that
\begin{equation}
\label{firscasetint}
[0.411185, \dim_{\cH}(J_{E_{s,q}})] \subset DS (\cf_{E_{s,q}}).
\end{equation}
Note that $\theta_{\cf_{E_{s,q}}}=1/2$, hence \cite[Theorem 6.3]{MU2} implies that 
\begin{equation}
\label{seccasetint}
[0, 1/2) \subset DS (\cf_{E_{s,q}}).
\end{equation}
Combining \eqref{firscasetint} and \eqref{seccasetint} we deduce that in the case when $s_1\geq \dim_{\cH}(J_{E_{s,q}})$, the system $\cf_{E_{s,q}}$ has full dimension spectrum.

We now consider the case when $s_1=0.597746<\dim_{\cH}(J_{E_{s,q}})$.  Since  \eqref{qsum2} holds for all $k \geq 5$, and \eqref{initicomp} implies that $P_{I_{s,q}(5)}(s_1) <0,$ we can apply Proposition \ref{keypropo2} to get:
\begin{equation}
\label{seccase}
[0.597746, \dim_{\cH}(J_{E_{s,q}})] \subset DS (\cf_{{E_{s,q}}}).
\end{equation}
Note now, that \eqref{qsum4}, \eqref{press2} and Proposition \ref{keypropo2} imply that
\begin{equation}
\label{firscasetint1}
[0.411185, 0.597746] \subset DS (\cf_{E_{s,q}}).
\end{equation}
Combining \eqref{seccase}, \eqref{firscasetint1} and \eqref{seccasetint} we deduce that  $\cf_{E_{s,q}}$ has full dimension spectrum. The proof is complete.
\end{proof}
Before proving Theorem \ref{progrespc} we need the following auxiliary lemma.
\begin{lm} 
\label{kesem}Let  $m \in \N$ and set $E_m=\{m, m+1, \dots\}$. The  system $\cf_{E_m}$ has full spectrum.
\end{lm}
\begin{proof}
If $m=1$ the result is due to Kesseb\"ohmer and Zhu \cite{KessebohmerM_ZhuS_2006}. We can thus assume that $m \geq 2$. We want to determine for which $k \in \N$
\begin{equation}
\label{sumkese}
\sum_{n=k+1}^\infty \|\f'_n\|^{h_m}_\infty \geq K_m^{2 h_m} \, \|\f'_{k}\|^{h_m}_\infty,
\end{equation}
where $h_m=\dim_{\cH}(J_{E_m})$ and $K_m=K_{E_m}$. Since $K_m \leq K_2=\exp(2/3)$ and $h_m <1$,  \cite[Lemma 6.17]{ChousionisV_LeykekhmanD_UrbanskiM_2018a} implies that if 
\begin{equation}
\label{sumkese1}
\sum_{n=k+1}^\infty \|\f'_n\|_\infty \geq \exp(2/3)^2 \, \|\f'_{k}\|_\infty
\end{equation}
holds, then \eqref{sumkese} holds as well. Observer that if
\begin{equation}
\label{condkese2}
\frac{k^2}{k+1} \geq \exp(4/3),
\end{equation}
then it follows easily that \eqref{sumkese1} holds. Since $f(x)=\frac{x^2}{x+1}$ is increasing for $x \geq 0$ and \eqref{condkese2} holds for $k=5$ we deduce that \eqref{sumkese} holds for $k \geq 5$. Hence Corollary \ref{keycoro2} implies that $\cf_{E_m}$ has full spectrum for all $m \geq 4$. Therefore it only remains to check the cases $m=2,3,4$. We provide the details for the case $m=2$, the other cases follow similarly. Let $I=\{2,3,4\}$. Using the code of Falk-Nussbaum from \cite{FalkRS_NussbaumRD_2016a}, and verifying our estimate using \texttt{IntLab} as in Section \ref{sec:dim}, we deduce that
$$\dim_{\cH}(J_{I}) \in [0.480695, 0.480697].$$
Hence if $s_1=0.481$, $P_{I}(s_1) \leq 0$. Moreover \eqref{sumkese} holds for $k \geq 5$ therefore Proposition \ref{keypropo2} implies that $[s_1, h_{m}] \subset DS(\cf_{E_2})$. Finally \cite[Theorem 6.3]{MU2} implies that $[0, 1/2) \subset DS (\cf_{E_{2}})$, because $\theta_{\cf_{E_{2}}}=1/2$. Joining these intervals we obtain that $\cf_{E_2}$ has full spectrum.
\end{proof}
We are now ready to prove Theorem \ref{progrespc} which we restate for convenience of the reader. 
\begin{thm} If $A \subset \N$ is any infinite arithmetic progression then $\cf_{A}$ has full spectrum.
\end{thm}
\begin{proof} Any infinite arithmetic progression is of the form $A=\{s+nq\}_{n \geq 0}$ where $s,q \in \N$. If $s \leq q$, then the conclusion follows by Theorem \ref{arithm1}.  The case $q=1$ follows by Lemma \ref{kesem}. The case $q=2$ follows by Corollary \ref{keycoro} since we have proved that the inequalities \eqref{evensum} and \eqref{oddsum} hold for all $k \geq 2$. Therefore we can assume that $q \geq 3$ and $s>q$. In that case there exists some $k \geq1$ and some $0<r \leq q$ such that $s=kq+r$. Hence $A=\{r+nq\}_{n \geq k}$. Since we have shown in the proof of Theorem \ref{arithm1} that  \eqref{qsum2} holds for all $k \geq 5$, Corollary \ref{keycoro2} implies that $\cf_{A}$ has full spectrum if $k \geq 4$. Now suppose that $k=1, 2, 3$ and let $I_k=\{r+kq, \dots, r+4q\}$. Note that since $q \geq 3$, Proposition \ref{bijpropo} implies that
\begin{equation}
\label{compdim}
\dim_{\cH}(J_{I_k}) \leq \dim_{\cH}(J_{I_1}) \leq \dim_{\cH}(J_{\{4, 7, 10, 13\}}).
\end{equation}
Using the code of Falk-Nussbaum from \cite{FalkRS_NussbaumRD_2016a}, and verifying our estimate using \texttt{IntLab} as in Section \ref{sec:dim}, we deduce that
$$\dim_{\cH}(J_{\{4, 7, 10, 13\}}) \in [0.3455682, 0.3455683].$$
Therefore by \eqref{compdim} we deduce that $\dim_{\cH}(J_{I_k}) <0.35$. Hence $P_{I_k}(0.35)<0$ and since \eqref{qsum2} holds for all $k \geq 5$, Proposition \ref{keypropo2} implies that 
$$[0.35, \dim_{\cH}(J_{A})] \subset DS(\cf_{A}).$$
Finally \cite[Theorem 6.3]{MU2} implies that $[0, 1/2) \subset DS (\cf_{A})$, because $\theta_{\cf_{A}}=1/2$. Therefore $\cf_A$ had full dimension spectrum. The proof is complete.
\end{proof}

\subsection{Primes.} In this section we will prove Theorem \ref{primes} which we restate below.
\begin{thm}
\label{primespectrum}
The system $\cf_{{prime}}$ has full spectrum.
\end{thm}
\begin{proof} We will now turn out attention to $\cf_{{prime}}$. Recall from Section \ref{sec:dim} that $\dim_{\cH}(J_{E_{prime}}) < 0.6752:=s$. We intend to use Corollary \ref{keycoro}. For $n \in \N$ let $p(n)$ be the $n$-th prime. Proposition \ref{distco} implies that $K_{prime}:=K=\exp(2/3)$. We will show that 
\begin{equation}
\label{primesum}\sum_{n=k+1}^\infty \left( \frac{1}{p(n)}\right)^{2s} \geq  \left( \frac{K}{p(k)}\right)^{2s}
\end{equation}
for all $k \geq 2$. By \cite{DusartP_1999}, for all $n \geq 5$
$$p(n) \leq n ( \ln(n)+\ln(\ln(n))).$$
Hence
\begin{equation}
\label{1stprimint}
\begin{split}
\sum_{n=k+1}^\infty {p(n)}^{-2s} &\geq \sum_{n=k+1}^\infty \left(n(\ln(n)+\ln(\ln(n)))\right)^{-2s}\\
&\geq \int_{k+1}^\infty \left( x(\ln(x)+\ln(\ln(x)))\right)^{-2s} dx.
\end{split}
\end{equation}
We will now show that
\begin{equation}
\label{loglogsqrt}
\ln(x)+\ln(\ln(x)) \leq \sqrt{x} \quad\mbox{ for all }x \geq 10 \ (\text{actually for all $x>0$)}.
\end{equation}
Let $f(x)=\sqrt{x}-\ln(x)-\ln(\ln(x))$. Then
$$f'(x)=\frac{(\sqrt{x}-2)\ln(x)-2}{2 x \ln (x)}:=\frac{g(x)}{2 x \ln (x)}.$$
Note that $g'(x)=\frac{\sqrt{x}(2+ \ln(x))-4}{2x}>0$ for all $x\geq 4$ and $g(10)>0$. Hence we deduce that $f'$ is increasing in $[10, \infty)$. Since $f(10)>0$ it follows that \eqref{loglogsqrt} holds.
Combining \eqref{1stprimint} and \eqref{loglogsqrt} we deduce that for $k \geq 9$
$$\sum_{n=k+1}^\infty p(n)^{-2s} \geq \int_{k+1}^\infty (x \sqrt{x})^{-2s} dx = \int_{k+1}^\infty x^{-3s} dx=\frac{(k+1)^{1-3s}}{3s-1}.$$
It is well known, see e.g. \cite[Lemma 1]{DusartP_1999}, that
$$p(n) \geq n \ln(n) \quad \mbox{ for }n \geq 2.$$
Hence, if 
\begin{equation}
\label{primesum2}
\frac{(k \ln (k))^{2s}}{(k+1)^{3s-1}} \geq \exp(2s/3) (3s-1)
\end{equation} holds for $k \geq 9$, then
\eqref{primesum} holds for $k \geq 9$ as well. Let $$h_p(x)=\ln(x)^{2s}\frac{x^{2s}}{(x+1)^{3s-1}}.$$
Notice that $2s>3s-1$ and consequently $h_p$ is increasing for $x >1$, as a product of two positive increasing functions. 
Moreover $h_p(9)> \exp(2s/3) (3s-1)$ hence $h_p(k) \geq \exp(2s/3) (3s-1)$ for all $k \geq 9$. Therefore we have proved that \eqref{primesum2}, and consequently \eqref{primesum}, hold for all $k \geq 9$. Using \texttt{Matlab} it is easy to verify that in fact \eqref{primesum} holds for the remaining $k=2,\dots,8$, hence Corollary \ref{keycoro2} implies that $\cf_{{prime}}$ has full spectrum.
\end{proof}

\subsection{Squares.} In this section we will prove Theorem \ref{squares} which we now restate.
\begin{thm}The system $\cf_{square}$ has full dimension spectrum.
\end{thm}
\begin{proof}
We have that $\cf_{square}:=\cf_{E}$ where $E=\{n^2:n \in \N\}$. We intend to use Proposition \ref{keypropo} and a bootstrapping argument as in the proof of \cite[Theorem 1.4]{ChousionisV_LeykekhmanD_UrbanskiM_2018a}. Let $s>1/4$. We will first investigate for which $k \in \N$
\begin{equation}
\label{squaresum}\sum_{n=k+1}^\infty \left( \frac{1}{n^2+1}\right)^{2s} \geq \left( \frac{2}{k^2+2}\right)^{2s}.
\end{equation}
By the integral test,
$$\sum_{n=k+1}^\infty \left( \frac{1}{n^2+1}\right)^{2s}  \geq \int_{k+1}^\infty (x^2+1)^{-2s} dx \geq 2^{-2s} \int_{k+1}^\infty x^{-4s} dx=4^{-s}\frac{(k+1)^{1-4s}}{4s-1}.$$
Hence if $k \in \N$ satisfies
\begin{equation}
\label{square2}
\frac{(k^2+2)^{2s}}{(k+1)^{4s-1}} \geq 16^s(4s-1),
\end{equation}
then it also satisfies \eqref{squaresum}. Let
$$h_{sq}(x)=\frac{(x^2+2)^{2s}}{(x+1)^{4s-1}},$$
and notice that it is increasing for $x \geq 1$. Now recall that from Section \ref{sec:dim} that $\dim_{\cH}(J_{E}) < 0.5766:=s_0$. By substitution it follows that \eqref{square2} holds for $k=8$ and $s=s_0$, and since $h_{sq}$ is increasing \eqref{square2} holds for all $k \geq 8$ and $s=s_0$. Therefore \eqref{squaresum}  holds for all $k \geq 8$ and $s=s_0$. Using \texttt{Matlab} we verify that \eqref{oddsum} holds for all $k \geq 3$ and $s=s_0$. 

We now consider the finite conformal system $\cS_{I(2)}$, where in this case $I(2)=\{1,4\}$ consists of the first two members of the alphabet $E$. Using the code of Falk-Nussbaum from \cite{FalkRS_NussbaumRD_2016a}, and verifying our estimate using \texttt{IntLab} as in Section \ref{sec:dim} we deduce that
$$\dim_{\cH}(J_{I(2)}) \in [0.411181, 0.411184].$$
Hence $P_{I(2)}(0.411184)<0$ and since \eqref{oddsum} holds for all $k \geq 3$ and $s=s_0$, Proposition \ref{keypropo} implies that
\begin{equation}
\label{square1est}
[0.411184, \dim_{\cH}(J_{E})] \subset DS(\cf_{square}).
\end{equation}

Now let $s_1=0.411184$. Working as previously and using the monotonicity of $h_{sq}$ we deduce that \eqref{square2} holds for all $k \geq 2$ and $s=s_1$. Hence \eqref{squaresum} holds for all $k \geq 2$ and $s=s_1$, so Corollary \ref{keycoro} implies that
\begin{equation}
\label{square2est}
[0, 0.411184] \subset DS(\cf_{square}).
\end{equation}
Combining \eqref{square1est} and \eqref{square2est} we deduce that $\cf_{square}$ has full spectrum. 
\end{proof}
\subsection{Powers.}
In this section we investigate the dimension spectrum of continued fractions whose entries are scaled powers. In particular we will study systems $\cf_{E}$ where $E$ is of the form $\{r^{-1}\lambda^n: n \in \N\}$ for $\lambda,r \in \N, \lambda \geq 2,$ and $r|\lambda$. Our first theorem asserts that the dimension spectrum of these systems always contains a non-trivial interval. Notice that Theorem \ref{powerintro} is an immediate corollary of the following theorem.
\begin{thm}
\label{powerint}
Let $\lambda,r \in \N, \lambda \geq 2,$ such that $r|\lambda$. Let $E_{r,\lambda^\ast}=\{r^{-1}\lambda^n: n \in \N\}$. Then there exists some $s(r,\lambda)>0$ such that
$$[0, \min\{s(r,\lambda), \dim_{\cH}(J_{E_{r,\lambda^\ast}})\}] \subset DS(\cf_{E_{r,\lambda^\ast}}).$$
\end{thm}
\begin{proof}
Let 
\begin{equation*} 
K:=K_{E_{\lambda^\ast}}=\begin{cases}\exp\left(\frac{2}{(\lambda/r)^2-1}\right),&\mbox{ if }r<\lambda,\\
4,&\mbox{ if }r=\lambda.
\end{cases}
\end{equation*}
 We aim to use Corollary \ref{keycoro2}, so we would like to investigate for which $k \geq 2$ and $s>0$,
\begin{equation*}
\sum_{n=k+1}^\infty \|\f'_{e_n}\|^s_\infty \geq K^{2s}\,\|\f'_{e_k}\|^{s},
\end{equation*}
or equivalently 
\begin{equation}
\label{powersum}
\sum_{n=k+1}^\infty \lambda^{-2ns} \geq K^{2s}\,\lambda^{-2ks}.
\end{equation}
We have
$$\sum_{n=k+1}^\infty \lambda^{-2ns} \geq \int_{k+1}^\infty \lambda^{-2s x} dx= \frac{\lambda^{-2(k+1)s}}{2 s\, \ln \lambda}.$$
Therefore if 
\begin{equation}
\label{power2}
\frac{\lambda^{-2(k+1)s}}{2 s\, \ln \lambda} \geq K^{2s} \lambda^{-2s k},
\end{equation}
holds for some $k$ and $s$, then \eqref{powersum} holds as well. Observe though, that \eqref{power2} is equivalent to
\begin{equation}
\label{power3}\frac{1}{2 s \ln \lambda} -(\lambda K)^{2s} \geq 0,\end{equation}
which does not depend on $k$. Let 
\begin{equation}
\label{rootpower}
f(s)=\frac{1}{2 s \ln \lambda} -(\lambda K)^{2s}.
\end{equation}
It follows that $f$ is decreasing in $(0, +\infty)$, $\lim_{s \ra 0} f(x)=+\infty$ and $\lim_{ s \ra +\infty}=-\infty$. Hence there exists a unique $s(\lambda)>0$ such that $f(s(\lambda))=0$ and $f(s)>0$ for all $s \in (0,s(\lambda))$. In other words \eqref{power3} holds for all $s \in (0, s(\lambda)]$, and consequently \eqref{powersum} holds for $k \geq 2$ and $s(\lambda)$. Hence by Corollary \ref{keycoro} $$[0, \min\{s(\lambda), \dim_{\cH}(J_{E_{\lambda^\ast}})\}] \subset DS(\cf_{E_{\lambda^\ast}}).$$ The proof is complete.
\end{proof}
It is easy to find numerically the roots of the function $f$, defined in \eqref{rootpower}, for given values of $\lambda$ and $r$. In the following corollary we collect the intervals contained in the spectrum of the systems $\cf_{E_{power\lambda }}$ for $\lambda=2,3,4,5$.
\begin{corollary} Let $E_{power \lambda}=\{\lambda^n\}_{n \in \N}$. Then
\begin{align*}
[0,0.3102] &\subset  DS(\cf_{power2}),\\
[0,0.2389] &\subset  DS(\cf_{power3}),\\
[0,0.1977]&\subset  DS(\cf_{power4}),\\
[0,0.1729]&\subset  DS(\cf_{power5}).\\
\end{align*}
\end{corollary}
The following lemma follows easily from Theorem \ref{kessenode}. We provide the details for completeness.
\begin{lm} Let $\lambda,r \in \N, \lambda \geq 2,$ such that $r|\lambda$. Let $E_{r,\lambda^\ast}=\{r^{-1}\lambda^n: n \in \N\}$. If there exist $0<s<r \leq \dim_{\cH}(J_{E_{r,\lambda^\ast}})$ such that
\begin{equation}
\label{powernowden}
 \frac{1}{2 t \ln \lambda}-K_{E_{r,\lambda^\ast}}^{-2t}<0\mbox{ for all } t\in (s,r),
\end{equation}
then $DS(\cS) \cap [s,r]$ is nowhere dense.
\end{lm}
\begin{proof} By Theorem \ref{kessenode}  it suffices to prove that 
\begin{equation}
\label{powersumnd}
\sum_{n=k+1}^\infty \lambda^{-2nt} < K_{E_{r,\lambda^\ast}}^{-2t}\,\lambda^{-2kt} \mbox{ for all }t \in (r,s).
\end{equation} 
For such $t$ we have
$$\sum_{n=k+1}^\infty \lambda^{-2nt} \leq \int_{k}^\infty \lambda^{-2t x} dx= \frac{\lambda^{-2kt}}{2 t\, \ln \lambda}.$$
Hence \eqref{powersumnd} holds if
$$\frac{1}{2 t \ln \lambda}-K_{E_{r,\lambda^\ast}}^{-2t}<0.$$
The proof is complete.
\end{proof}
We will now prove Theorem \ref{50100intro} which we restate in a more precise way below.
\begin{thm} Let $E_{50,100^{\ast}}=\{2,200,20000,\dots\}$. There exist $s_1,s_2$ such that $0<s_1<s_2<h:=\dim_{\cH}(J_{E_{50,100^{\ast}}})$ and
\begin{enumerate}[label=(\roman*)]
\item \label{inter} $[0,s_1] \subset DS(\cf_{E_{{50,100^{\ast}}}})$,
\item \label{nowden} $[s_2, h] \cap DS(\cf_{E_{{50,100^{\ast}}}})$ is nowhere dense.
\end{enumerate}
\end{thm}
\begin{proof} Arguing as in Section \ref{sec:dim} (i.e. using numerical approximation based on the method of Falk-Nussbaum and estimating the error using Theorem \ref{hein}) we obtain that
$$h=\dim_{\cH}(J_{E_{50,100^{\ast}}}) \in [0.160397,0.160398].$$
The existence of $s_1$ follows by Theorem \ref{powerint}, and solving the equation \eqref{power3} numerically we can choose $s_1 = 0.058557$. 

We will use Lemma \ref{powernowden} in order to determine $s_2$. In the case of $E_{50,100^{\ast}}$, Proposition \ref{distco} implies that $K:=K_{E_{50,100^{\ast}}}=\exp(2/3)$. Let 
$$f(t)=\frac{1}{2 t \ln 100}-\exp(-4t/3), \quad t >0.$$
Since $f'(t)=-\frac{1}{ t^2 2 \ln 100}+\frac{4}{3}\exp (-4t/3)$ and, 
$\frac{4}{3}\exp (-4t/3) < 4/3$ for $t >0$,
we deduce that $f$ is decreasing in $(0, s^{\ast})$ where
$$s^{\ast}=\sqrt{\frac{3}{8 \ln100}}\approx 0.285.$$
Since $f(s^{\ast})<0$ and $\lim_{t \ra^+ 0}f(t)=+\infty$ we deduce that there exists a unique $s_2<s^*$ such that $f(s_2)=0$.
Solving the equation $f(t)=0$ numerically in $(0,s^\ast)$ we derive that $s_2\approx 0.12894$. Hence we have shown that  $[0, 0.058557] \subset DS(\cf_{E_{50,100^{\ast}}})$  and $[0.128941, h] \cap DS(\cf_{E_{50,100^{\ast}}})$ is nowhere dense. The proof is complete. 
\end{proof}
\subsection{Lacunary sequences.} We finally consider continued fractions whose entries form a lacunary sequence. Indicatively we consider the lacunary sequence $E_{3lac}=\{3^{n^2}\}_{n \in \N}$ and we prove that the corresponding continued fractions system contains a non-trivial nowhere dense part. The same scheme works for several other lacunary sequences as well.
\begin{thm}
\label{lacspec} Let $E_{3lac}=\{3^{n^2}\}_{n \in \N}$. There exists some $t_0<\dim_{\cH}(J_{E_{3lac}})$ such that $DS(\cf_{3lac}) \cap [t_0,\dim_{\cH}(J_{E_{3lac}})]$ is nowhere dense.
\end{thm}
\begin{proof} Arguing as in Section \ref{sec:dim} (i.e. using numerical approximation based on the method of Falk-Nussbaum and estimating the error using Theorem \ref{hein}) we obtain that
$$h=\dim_{\cH}(J_{E_{3lac}}) \in [0.15565551, 0.15565552].$$
Let $t \in (0,h]$ and $k \in \N$. We have
$$\sum_{n=k+1}^\infty \|\f'_{e_n}\|_{\infty}^{t} \leq \int_{k}^\infty (3^{2t})^{-x^2}dx=\int_{k}^\infty e^{-\alpha_t x^2} dx$$
where $\alpha_t=2t \ln3$. Integrating by parts we get
\begin{equation*}
\begin{split}
\int_{k}^\infty e^{-\alpha_t x^2} dx&=\frac{1}{2 \alpha_t} \int_{k}^\infty \frac{1}{x} (-e^{-\alpha_t x^2})'dx \\
&=\frac{e^{-\alpha_t k^2}}{2 \alpha_t k}-\frac{1}{2 \alpha_t} \int_{k}^\infty \frac{e^{-\alpha_t x^2}}{x^2}dx 
\leq \frac{e^{-\alpha_t k^2}}{2 \alpha_t k}=\frac{3^{-2tk^2}}{4 t k\ln3 }.
\end{split}
\end{equation*}
In the case of $E_{3lac}$, Proposition \ref{distco} implies that $K_{3lac}:=K=\exp(1/4)$. Note that
$$\frac{3^{-2tk^2}}{4 t k\ln3 } < K^{-2t} \|\f'_{e_k}\|^{t}_{\infty}=K^{-2t} 3^{-2tk^2}$$
is equivalent to $\frac{\exp(t/2)}{4 t \ln3} <k.$ Let 
$$f(t)=\frac{\exp(t/2)}{4 t \ln3}, \quad t>0.$$
Since $f$ is decreasing for $t \in (0,2)$ and for $t_0=1.42$ we have that $f(t_0)<1.8$ we deduce that $f(t)<k$ for all $t \in [t_0, 2)$ and $k \geq 2$. Therefore
\begin{equation}
\label{sumgeq2}
\sum_{n=k+1}^\infty \|\f'_{e_n}\|_{\infty}^{t}<K^{-2t} \|\f'_{e_k}\|^{t}_{\infty}, \quad \mbox{ for all }k \geq 2, t\in [t_0, 2).
\end{equation}
Now let 
$$I(t)=\int_{1}^\infty e^{-\alpha_t x^2} dx-(3\exp(1/4))^{-2t}.$$
Taking derivatives we get that
$$I'(t)=-2 \ln 3 \int_1^\infty x^2 e^{-2 t \ln3 \, x^2}dx+2 \ln(3 \exp(1/4))3^{-2t}.$$
Note that in the interval $t \in [t_0, h],$
$$I'(t) \leq 2(\ln(3 \exp(1/4))3^{-2t_0}-\ln 3 \int_1^\infty x^2 e^{-2 h \ln3 \, x^2}dx):=E.$$
The quantity $E$ is easily computable and it turns out that $E\approx-2.27867$, hence $I'(t) <0$ for $t \in [t_0, h]$. Since $I(t_0)<0$ (computed using \texttt{Matlab}) we deduce that
\begin{equation}
\label{gaus1}
I(t) <0 \mbox{ for all }t \in [t_0,h].
\end{equation}
Combining \eqref{sumgeq2} and \eqref{gaus1} we obtain that
\begin{equation}
\label{sumgeq3}
\sum_{n=k+1}^\infty \|\f'_{e_n}\|_{\infty}^{t}<K^{-2t} \|\f'_{e_k}\|^{t}_{\infty}, \quad \mbox{ for all }k \geq 1, t\in [t_0, h].
\end{equation}
Therefore Theorem \ref{kessenode} implies that $[t_0,h] \cap DS(\cf_{E_{3lac}})$ is nowhere dense. The proof is complete.
\end{proof}
\bibliography{IFS_1D}

\begin{bibdiv}
\begin{biblist}

\bib{bedford}{inproceedings}{
      author={Bedford, T.},
       title={Hausdorff dimension and box dimension in self-similar sets},
        date={1988},
   booktitle={Proceedings of the {C}onference: {T}opology and {M}easure, {V}
  ({B}inz, 1987)},
      series={Wissensch. Beitr.},
   publisher={Ernst-Moritz-Arndt Univ., Greifswald},
       pages={17\ndash 26},
      review={\MR{1029553}},
}

\bib{zar}{article}{
      author={Bourgain, J.},
      author={Kontorovich, A.},
       title={On {Z}aremba's conjecture},
        date={2014},
        ISSN={0003-486X},
     journal={Ann. of Math. (2)},
      volume={180},
      number={1},
       pages={137\ndash 196},
         url={https://doi.org/10.4007/annals.2014.180.1.3},
      review={\MR{3194813}},
}

\bib{bu1}{article}{
      author={Bumby, R.~T.},
       title={Hausdorff dimensions of {C}antor sets},
        date={1982},
        ISSN={0075-4102},
     journal={J. Reine Angew. Math.},
      volume={331},
       pages={192\ndash 206},
         url={https://doi.org/10.1515/crll.1982.331.192},
      review={\MR{647383}},
}

\bib{bu2}{incollection}{
      author={Bumby, R.~T.},
       title={Hausdorff dimension of sets arising in number theory},
        date={1985},
   booktitle={Number theory ({N}ew {Y}ork, 1983--84)},
      series={Lecture Notes in Math.},
      volume={1135},
   publisher={Springer, Berlin},
       pages={1\ndash 8},
         url={https://doi.org/10.1007/BFb0074599},
      review={\MR{803348}},
}

\bib{ChousionisV_LeykekhmanD_UrbanskiM_2018a}{article}{
      author={{Chousionis}, V.},
      author={{Leykekhman}, D.},
      author={{Urba{\'n}ski}, M.},
       title={{The dimension spectrum of graph directed Markov systems}},
        date={2018-02},
     journal={ArXiv e-prints},
      eprint={1802.01125},
}

\bib{CTU}{article}{
      author={{Chousionis}, V.},
      author={{Tyson}, J.~T.},
      author={{Urba\'nski}, M.},
       title={{Conformal graph directed Markov systems on Carnot groups}},
     journal={Memoirs of the AMS},
}

\bib{cu1}{article}{
      author={Cusick, T.~W.},
       title={Continuants with bounded digits},
        date={1977},
        ISSN={0025-5793},
     journal={Mathematika},
      volume={24},
      number={2},
       pages={166\ndash 172},
         url={https://doi.org/10.1112/S0025579300009050},
      review={\MR{0472721}},
}

\bib{cu2}{article}{
      author={Cusick, T.~W.},
       title={Continuants with bounded digits. {II}},
        date={1978},
        ISSN={0025-5793},
     journal={Mathematika},
      volume={25},
      number={1},
       pages={107\ndash 109},
         url={https://doi.org/10.1112/S002557930000930X},
      review={\MR{0498413}},
}

\bib{cu3}{article}{
      author={Cusick, T.~W.},
       title={Continuants with bounded digits. {III}},
        date={1985},
        ISSN={0026-9255},
     journal={Monatsh. Math.},
      volume={99},
      number={2},
       pages={105\ndash 109},
         url={https://doi.org/10.1007/BF01304191},
      review={\MR{781688}},
}

\bib{cufla}{book}{
      author={Cusick, T.~W.},
      author={Flahive, M.~E.},
       title={The {M}arkoff and {L}agrange spectra},
      series={Mathematical Surveys and Monographs},
   publisher={American Mathematical Society, Providence, RI},
        date={1989},
      volume={30},
        ISBN={0-8218-1531-8},
         url={https://doi.org/10.1090/surv/030},
      review={\MR{1010419}},
}

\bib{DusartP_1999}{article}{
      author={Dusart, P.},
       title={The {$k$}th prime is greater than {$k(\ln k+\ln\ln k-1)$} for
  {$k\geq2$}},
        date={1999},
        ISSN={0025-5718},
     journal={Math. Comp.},
      volume={68},
      number={225},
       pages={411\ndash 415},
         url={https://doi.org/10.1090/S0025-5718-99-01037-6},
      review={\MR{1620223}},
}

\bib{FalkRS_NussbaumRD_2016b}{article}{
      author={{Falk}, R.~S.},
      author={{Nussbaum}, R.~D.},
       title={{C\^{}m Eigenfunctions of Perron-Frobenius Operators and a New
  Approach to Numerical Computation of Hausdorff Dimension}},
        date={2016-01},
     journal={ArXiv e-prints},
      eprint={1601.06737},
}

\bib{FalkRS_NussbaumRD_2016a}{article}{
      author={{Falk}, R.~S.},
      author={{Nussbaum}, R.~D.},
       title={{$C^m$ Eigenfunctions of Perron-Frobenius Operators and a New
  Approach to Numerical Computation of Hausdorff Dimension: Applications in
  $\mathbb{R}^1$}},
        date={2016-12},
     journal={ArXiv e-prints},
      eprint={1612.00870},
}

\bib{good}{article}{
      author={Good, I.~J.},
       title={The fractional dimensional theory of continued fractions},
        date={1941},
     journal={Proc. Cambridge Philos. Soc.},
      volume={37},
       pages={199\ndash 228},
      review={\MR{0004878}},
}

\bib{HeinemannSM_UrbanskiM_2002a}{article}{
      author={Heinemann, S.},
      author={Urba\'nski, M.},
       title={Hausdorff dimension estimates for infinite conformal {IFS}s},
        date={2002},
        ISSN={0951-7715},
     journal={Nonlinearity},
      volume={15},
      number={3},
       pages={727\ndash 734},
         url={https://doi.org/10.1088/0951-7715/15/3/312},
      review={\MR{1901102}},
}

\bib{hen1}{article}{
      author={Hensley, D.},
       title={The {H}ausdorff dimensions of some continued fraction {C}antor
  sets},
        date={1989},
        ISSN={0022-314X},
     journal={J. Number Theory},
      volume={33},
      number={2},
       pages={182\ndash 198},
         url={https://doi.org/10.1016/0022-314X(89)90005-X},
      review={\MR{1034198}},
}

\bib{hen2}{article}{
      author={Hensley, D.},
       title={The {H}ausdorff dimensions of some continued fraction {C}antor
  sets},
        date={1989},
        ISSN={0022-314X},
     journal={J. Number Theory},
      volume={33},
      number={2},
       pages={182\ndash 198},
         url={https://doi.org/10.1016/0022-314X(89)90005-X},
      review={\MR{1034198}},
}

\bib{hentex}{article}{
      author={Hensley, D.},
       title={A polynomial time algorithm for the {H}ausdorff dimension of
  continued fraction {C}antor sets},
        date={1996},
        ISSN={0022-314X},
     journal={J. Number Theory},
      volume={58},
      number={1},
       pages={9\ndash 45},
         url={https://doi.org/10.1006/jnth.1996.0058},
      review={\MR{1387719}},
}

\bib{jar}{article}{
      author={Jarnik, V.},
       title={Zur metrischen theorie der diophantischen approximation},
        date={1928},
     journal={Prace Mat. Fiz.},
      volume={36},
       pages={91\ndash 106},
}

\bib{jenktexan}{article}{
      author={Jenkinson, O.},
       title={On the density of {H}ausdorff dimensions of bounded type
  continued fraction sets: the {T}exan conjecture},
        date={2004},
        ISSN={0219-4937},
     journal={Stoch. Dyn.},
      volume={4},
      number={1},
       pages={63\ndash 76},
         url={https://doi.org/10.1142/S0219493704000900},
      review={\MR{2069367}},
}

\bib{JenkinsonO_PollicottM_2001}{article}{
      author={Jenkinson, O.},
      author={Pollicott, M.},
       title={Computing the dimension of dynamically defined sets: {$E_2$} and
  bounded continued fractions},
        date={2001},
        ISSN={0143-3857},
     journal={Ergodic Theory Dynam. Systems},
      volume={21},
      number={5},
       pages={1429\ndash 1445},
         url={https://doi.org/10.1017/S0143385701001687},
      review={\MR{1855840}},
}

\bib{JenkinsonO_PollicottM_2018}{article}{
      author={Jenkinson, O.},
      author={Pollicott, M.},
       title={Rigorous effective bounds on the {H}ausdorff dimension of
  continued fraction {C}antor sets: a hundred decimal digits for the dimension
  of {$E_2$}},
        date={2018},
        ISSN={0001-8708},
     journal={Adv. Math.},
      volume={325},
       pages={87\ndash 115},
         url={https://doi.org/10.1016/j.aim.2017.11.028},
      review={\MR{3742587}},
}

\bib{KessebohmerM_ZhuS_2006}{article}{
      author={Kesseb\"ohmer, M.},
      author={Zhu, S.},
       title={Dimension sets for infinite {IFS}s: the {T}exan conjecture},
        date={2006},
        ISSN={0022-314X},
     journal={J. Number Theory},
      volume={116},
      number={1},
       pages={230\ndash 246},
         url={https://doi.org/10.1016/j.jnt.2005.04.002},
      review={\MR{2197868}},
}

\bib{MU1}{article}{
      author={Mauldin, R.~D.},
      author={Urba\'nski, M.},
       title={Dimensions and measures in infinite iterated function systems},
        date={1996},
     journal={Proc.\ London Math.\ Soc.\ (3)},
      volume={73},
      number={1},
       pages={105\ndash 154},
}

\bib{MU2}{article}{
      author={Mauldin, R.~D.},
      author={Urba\'nski, M.},
       title={Conformal iterated function systems with applications to the
  geometry of conformal iterated function systems},
        date={1999},
     journal={Trans. Amer. Math. Soc.},
      volume={351},
       pages={4995\ndash 5025},
}

\bib{MauldinRD_UrbanskiM_2003a}{book}{
      author={Mauldin, R.~D.},
      author={Urba\'nski, M.},
       title={Graph directed {M}arkov systems},
      series={Cambridge Tracts in Mathematics},
   publisher={Cambridge University Press, Cambridge},
        date={2003},
      volume={148},
        ISBN={0-521-82538-5},
         url={https://doi.org/10.1017/CBO9780511543050},
        note={Geometry and dynamics of limit sets},
      review={\MR{2003772}},
}

\bib{McMullenC_1998}{article}{
      author={McMullen, C.~T.},
       title={Hausdorff dimension and conformal dynamics. {III}. {C}omputation
  of dimension},
        date={1998},
        ISSN={0002-9327},
     journal={Amer. J. Math.},
      volume={120},
      number={4},
       pages={691\ndash 721},
  url={http://muse.jhu.edu/journals/american_journal_of_mathematics/v120/120.4mcmullen.pdf},
      review={\MR{1637951}},
}

\bib{ruelle}{book}{
      author={Ruelle, D.},
       title={Thermodynamic formalism},
      series={Encyclopedia of Mathematics and its Applications},
   publisher={Addison-Wesley Publishing Co., Reading, Mass.},
        date={1978},
      volume={5},
        ISBN={0-201-13504-3},
        note={The mathematical structures of classical equilibrium statistical
  mechanics, With a foreword by Giovanni Gallavotti and Gian-Carlo Rota},
      review={\MR{511655}},
}

\bib{RumpS_2010}{article}{
      author={Rump, S.~M.},
       title={Verification methods: rigorous results using floating-point
  arithmetic},
        date={2010},
        ISSN={0962-4929},
     journal={Acta Numer.},
      volume={19},
       pages={287\ndash 449},
         url={https://doi.org/10.1017/S096249291000005X},
      review={\MR{2652784}},
}

\bib{StoerJ_BulirschR_1993}{book}{
      author={Stoer, J.},
      author={Bulirsch, R.},
       title={Introduction to numerical analysis},
     edition={Second},
      series={Texts in Applied Mathematics},
   publisher={Springer-Verlag, New York},
        date={1993},
      volume={12},
        ISBN={0-387-97878-X},
         url={https://doi.org/10.1007/978-1-4757-2272-7},
        note={Translated from the German by R. Bartels, W. Gautschi and C.
  Witzgall},
      review={\MR{1295246}},
}

\bib{UZ}{article}{
      author={Urba\'nski, M.},
      author={Zdunik, A.},
       title={Continuity of the {H}ausdorff measure of continued fractions and
  countable alphabet iterated function systems},
        date={2016},
        ISSN={1246-7405},
     journal={J. Th\'eor. Nombres Bordeaux},
      volume={28},
      number={1},
       pages={261\ndash 286},
         url={http://jtnb.cedram.org/item?id=JTNB_2016__28_1_261_0},
      review={\MR{3464621}},
}

\end{biblist}
\end{bibdiv}
\bibliographystyle{siam}

\end{document}